\documentclass[a4paper,reqno]{amsart}

\usepackage{amsmath,amssymb,amsthm}
\usepackage{cite}
\usepackage{xcolor}
\usepackage{mathtools}
\usepackage{enumerate}
\usepackage{nicefrac}
\usepackage{bm}
\usepackage{a4wide}
\usepackage{booktabs}
\usepackage{algorithm,algorithmicx,algpseudocode}
\usepackage[pdftex]{graphicx}
\usepackage{color}

\def\F{\mathcal F}

\def\M{\mathcal M}

\newcommand{\jump}[1]{\lbrack\!\lbrack#1\rbrack\!\rbrack} 
\newcommand{\dif}{\,\text{d}}
\renewcommand{\L}{\mathsf{L}}

\renewcommand{\phi}{\varphi}

\newcommand{\tol}{\mathtt{tol}}
\newcommand{\esssup}{\mathop{\mathrm{ess\,sup}}}

\newcommand{\res}{{\mathtt{res}}}
\newcommand{\proj}{{\mathtt{proj}}}

\newtheorem{Remark}[equation]{Remark}

\newtheorem{theorem}[equation]{Theorem}
\newtheorem{proposition}[equation]{Proposition}
\newtheorem{corollary}[equation]{Corollary}
\newtheorem{lemma}[equation]{Lemma}

\numberwithin{equation}{section}

\title[$hp$-Adaptive Galerkin Time Stepping Methods for Nonlinear IVPs]{$hp$-Adaptive Galerkin Time Stepping Methods for \\ Nonlinear Initial Value Problems}
\author{Irene Kyza}
\address{Department of Mathematics, University of Dundee, Nethergate, Dundee, DD1 4HN, United Kingdom}
\email{ikyza@dundee.ac.uk}

\author{Stephen Metcalfe \and Thomas P.~Wihler}
\address{Mathematisches Institut, Universit\"at Bern, Sidlerstr.~5, CH-3012 Bern, Switzerland}
\email{smetcalfephd@gmail.com \and wihler@math.unibe.ch}

\thanks{The authors acknowledge the support of the Swiss National Science Foundation (SNF)}

\begin{document}

\begin{abstract}
This work is concerned with the derivation of an \emph{a posteriori} error estimator for Galerkin approximations to nonlinear initial value problems with an emphasis on finite-time existence in the context of blow-up. The stucture of the derived estimator leads naturally to the development of both $h$ and $hp$ versions of an adaptive algorithm designed to approximate the blow-up time. The adaptive algorithms are then applied in a series of numerical experiments, and the rate of convergence to the blow-up time is investigated.

\end{abstract}

\keywords{Initial value problems in Hilbert spaces, Galerkin time stepping schemes, high-order methods, blow-up singularities.}

\subjclass[2010]{65J08, 65L05, 65L60}

\maketitle


\section{Introduction}

In this paper we focus on continuous Galerkin (cG) and discontinuous Galerkin (dG) time stepping discretizations as applied to abstract 
initial value problems of the form
\begin{align}
  u'(t) &= \F(t,u(t)), \quad t\in(0,T),\qquad\qquad u(0) = u_0.\label{eq:1}
\end{align}
Here, $u:(0,T)\to H$, for some~$T>0$, denotes the unknown solution with values in a real Hilbert space~$H$ with inner product ~$(\cdot,\cdot)_H$ and induced norm~$\|\cdot\|_H$. 
The initial value $u_0 \in H$ prescribes the value of the solution at $t=0$, and~$\F:[0,T]\times H\to H$ is a possibly nonlinear operator. We emphasize that we include the case of~$\F$ being {\em unbounded} in the sense that
\begin{equation}\notag
\frac{\|\F(t,x)\|_H}{\|x\|_H}\to\infty\text{ as }\|x\|_H\to\infty,\qquad 0\le t\le T.
\end{equation}
Note that the existence interval for solutions may be arbitrarily small even for smooth~$\F$. Indeed, for certain data the solution of~\eqref{eq:1} can become unbounded in finite time, i.e.,
\[
\|u(t)\|_H<\infty\text{ for }0<t<T_\infty,\qquad \qquad \lim_{t\nearrow T_\infty}\|u(t)\|_H = \infty.
\]
This effect is commonly termed \emph{finite-time blow-up} or sometimes just \emph{blow-up} and the quantity $T_{\infty}$ is called the \emph{blow-up time}.

The main contributions of this paper are as follows:
\begin{enumerate}[$\bullet$]
\item The derivation of conditional \emph{a posteriori} error bounds for $hp$-cG and $hp$-dG approximations to the nonlinear initial value problem \eqref{eq:1}.
\item The design of efficient adaptive algorithms that lead to accurate approximation of the blow-up time in the case where problem \eqref{eq:1} exhibits finite time blow-up.
\end{enumerate}
To the best of our knowledge, this is the first time that an adaptive algorithm has been developed for $hp$-cG and $hp$-dG  time-stepping methods based on rigorous \emph{a posteriori} error control for problems of the form \eqref{eq:1}. The adaptive procedure that we propose includes both $h$-adaptive and $hp$-adaptive variants. Indeed, one of the contributions of this paper is to motivate how we can choose effectively between $h$- or $p$-adaptivity locally while taking into account the possible singular behavior of the problem under consideration. In this sense, these results extend the $h$-adaptive algorithm analyzed for some special cases of \eqref{eq:1} and Euler-type time discretizations in  \cite{CGKM15}. In particular, the inclusion of higher order time-stepping schemes and $hp$-adaptivity allows us to overcome key limitations encountered in \cite{CGKM15}. 

\emph{A posteriori} error estimators for linear problems tend to be \emph{unconditional}, that is, they always hold independent of the problem data and the size of the discretization parameters.  For nonlinear problems, the situation is more complicated since the existence of a solution to an appropriate error equation (and, thus, of an error bound) usually requires that either the data or the discretization parameters are sufficiently small. As a result, \emph{a posteriori} error estimators for nonlinear problems tend to be \emph{conditional}, that is, they only hold provided that an \emph{a posteriori} verifiable condition (which can be either explicit or implicit) is satisfied. For nonlinear time-dependent problems, there are two commonly used approaches for deriving conditional \emph{a posteriori} error bounds: continuation arguments, cf. \cite{CGKM15, KNS04}, and fixed point arguments, cf. \cite{KM01, CM08}. The \emph{a posteriori} error bounds that we derive here are obtained by utilizing a local continuation argument.

Galerkin time stepping methods for initial value problems are based on weak formulations and for both the cG and dG time stepping schemes, the test spaces consist of polynomials 
that are discontinuous at the time nodes. In this way, the discrete Galerkin formulations decouple into local problems on each time step and the discretizations can therefore be 
understood as implicit one-step schemes. In the literature, Galerkin time stepping schemes have been extensively analyzed for ordinary differential equations (ODEs), cf. ~\cite{rannacher,DeDu86,DelfourHagerTrochu81,Estep95,EstepFrench94,johnson}. 

A key feature of Galerkin time stepping methods is their great flexibility with respect to the size of the time steps and the local approximation orders which lends these schemes well
to an $hp$-framework. The $hp$-versions of the cG and dG time stepping schemes were introduced and analyzed in the works~\cite{SchoetzauSchwabDGODE,SchoetzauSchwab00,ScWi10,Wihler05}.
In particular, in the articles~\cite{SchoetzauSchwabDGODE,Wihler05} which focus on initial value problems with uniform Lipschitz nonlinearities, 
the use of the Banach fixed point theorem made it possible to prove existence and uniqueness results for the discrete Galerkin solutions which are independent of the local approximation 
orders; these results have been extended to discrete Peano-type existence results in the context of more general nonlinearities in~\cite{JanssenWihler:15}. We emphasize that the $hp$-approach is well known for its ability to approximate smooth solutions with possible local singularities at high algebraic or even exponential rates of convergence; see, e.g., \cite{BrunnerSchoetzau06,SchoetzauSchwab00,SchoetzauSchwab01,Gerdes} for the numerical approximation of problems with start-up singularities. In light of this, a main aim of this paper is to establish through numerical experiments whether or not $hp$-refinement, utilizing the derived \emph{a posteriori} error estimator, can lead to exponential convergence towards the blow-up time for the case where \eqref{eq:1} exhibits finite time blow-up.

\subsection*{Outline} The remainder of our article is organized as follows. In Section 2, we introduce the $hp$-cG and $hp$-dG time stepping schemes while  in Section 3 we develop \emph{a posteriori} error bounds for these schemes. We design $h$ as well as $hp$ version adaptive algorithms to approximate the blow-up time in Section 4 before applying them to some numerical experiments in Section 5. Finally, we draw conclusions and comment on possible future research in Section 6.

\subsection*{Notation} Let $H$ denote a real Hilbert space with inner product $(\cdot,\cdot)_H$ and induced norm $\|\cdot\|_H$ as before. Given an interval $I=(a,b)$, the Bochner space $C(\bar{I};H)$ consists of all functions $u:\bar I\to H$ that are continuous on $\bar{I}$ with values in $H$. Moreover, for~$1\le p\le\infty$, we define the norm
\[
\displaystyle\|u\|_{L^p(I;H)}:=\begin{cases}
\displaystyle\left(\int_I\|u(t)\|^p_H\dif t\right)^{\nicefrac{1}{p}}, \qquad &1\le p<\infty,\\[2ex]
\displaystyle \esssup_{t\in I}\|u(t)\|_H,&p=\infty,
\end{cases}
\]
and we let $L^p(I;H)$ be the space of all measurable functions $u:I\to H$ such that the corresponding norm is bounded. Note that~$L^2(I;H)$ is a Hilbert space with inner product 
and norm given by
\[
(u,v)_{L^2(I;H)}:=\int_I(u(t),v(t))_H\dif t,\qquad\text{and}\qquad \|u\|_{L^2(I;H)}:=
\left(\int_I\|u(t)\|^2_H\dif t\right)^{\nicefrac{1}{2}},
\]
resepctively.


\section{Galerkin Time Discretizations}\label{disc}

In this section, we briefly recall the $hp$-cG and $hp$-dG time stepping schemes for the discretisation of~\eqref{eq:1}. To this end, on the open interval~$I=(0,T)$, we introduce a set of time nodes, $0 := t_0 < t_1 < \cdots < t_{M-1} < t_M := T$, which define a time partition $\M:=\{I_m\}_{m=1}^M$ of~$I$ into~$M$ open time intervals~$I_m:=(t_{m-1},t_m)$, $m=1,\ldots,M$. The length $k_m := t_m - t_{m-1}$ (which may be variable) of the time interval $I_m$ is called the time step length. Furthermore, to each interval $I_m$ we associate a polynomial degree~$r_m \in \mathbb{N}_0$ which takes the role of a local approximation order. Then, given a (real) Hilbert space~$X\subseteq H$ and some~$r\in\mathbb{N}_0$, the set
\[
\mathcal{P}^{r}(J;X):=\left\{p\in C(\bar J;X):\,p(t)=\sum_{i=0}^rx_it^i,\, x_i\in X\right\}\!,
\]
signifies the space of all polynomials of degree at most~$r$ on an interval~$J\subset\mathbb{R}$ with values in~$X$. 

In practical computations, the Hilbert space~$H$ on which~\eqref{eq:1} is based will typically be replaced by a finite-dimensional subspace~$H_m\subset H$ on each interval~$I_m$, $1\le m\le M$. In this context, it is useful to define the \emph{$H$\!-orthogonal projection} ~$\pi_m:\,H\to H_m$ given by
\begin{equation}\notag
x\mapsto\pi_m x,\qquad(x-\pi_mx,y)_H=0\qquad\forall y\in H_m.
\end{equation}

\subsection{$hp$-cG Time Stepping}

With the above definitions, the (fully-discrete) $hp$-cG time marching scheme is iteratively given as follows: For $1\le m\le M$, we seek $U|_{I_m}\in\mathcal{P}^{r_m}(I_m;H_m)$ such that
\begin{equation}
\begin{split}
 \int_{I_m} \bigg(\frac{\dif U}{\dif t},V\bigg)_H\dif t&=\int_{I_m} (\F(t,U),V)_H  \dif t \qquad\forall V\in\mathcal{P}^{r_m-1}(I_m;H_m),
  \label{eq:cG}
  \end{split}
\end{equation}
with the initial condition $U(t_{0}) =  \pi_1 u_0$ for $m=1$, and $U(t_{m-1}) =  \pi_m U(t_{m-1}^-)$ for $m \geq 2$; here, the one-sided limits of a piecewise continuous function $U$ at each time node $t_m$ are given by
\[
U(t_m^+) \coloneqq \lim_{s\searrow 0} U(t_m+s), \qquad\qquad
U(t_m^-) \coloneqq \lim_{s\searrow 0} U(t_m-s).
\]

Note that in order to enforce the initial condition on each time step, the local trial space has one degree of freedom more than the local test space in the $hp$-cG scheme. Furthermore, if~$H_1=\ldots=H_M$, we remark that the $hp$-cG solution~$U$ is globally continuous on~$[0,T]$. 

The strong form of \eqref{eq:cG} on ${I}_m$ is
\begin{equation}
\begin{split}
 U'(t) = \Pi^{r_m - 1}_m \F(t,{U}(t)),
  \label{eq:cGstrongsuboptimal}
\end{split}
\end{equation}
where $\Pi^{r_m - 1}_m$ denotes the $L^2$-projection operator onto the space $\mathcal{P}^{r_m-1}(I_m;H_m)$; see~\cite{JanssenWihler:15} for details. Whilst the strong form~\eqref{eq:cGstrongsuboptimal} can be exploited for the purpose of deriving \emph{a posteriori} error estimates, it is well known that employing such a straightfoward approach leads to suboptimal error estimates, cf. \cite{AkrivisMakridakisNochetto:09}. This issue will be addressed in the derivation of our error bound.

\subsection{$hp$-dG Time Stepping} 

The (fully-discrete) $hp$-dG time marching scheme is iteratively given as follows: For ~$1\le m\le M$, we seek $U|_{I_m} \in \mathcal{P}^{r_m}(I_m;H_m)$ such that
\begin{equation}
\begin{split}
  \int_{I_m} \bigg(\frac{\dif U}{\dif t},V\bigg)_H  \dif t 
  &+  ({\jump{U}_{m-1}}, V(t^+_{m-1}))_H
   = 
  \int_{I_m}(\F(t,U),V)_H \dif t\quad\forall V\in\mathcal{P}^{r_m}(I_m;H_m),
  \label{eq:dG}
  \end{split}
\end{equation}
where the discontinuity jump of $U$ at $t_m$, $0\le m \le M-1$, is given by ${\jump{U}_m} := U(t_m^+) - U(t_m^-)$ with~$U(t^-_0) := u_0$. We emphasize that, in contrast to the continuous Galerkin formulation, the trial and test spaces are the same for the dG scheme; this is due to the fact that the initial values are weakly imposed (by means of an upwind flux) on each time interval.

In order to find the strong formulation of ~\eqref{eq:dG}, we require the use of a \emph{lifting operator}. More precisely, given some real Hilbert space~$X\subseteq H$, we define~$\L^{r_m}_m:\,X\to\mathcal{P}^{r_m}(I_m;X)$ for $1\le m\le M$ uniquely through 
\begin{equation}\label{eq:liftingop}
z\mapsto \L^{r_m}_m z,\qquad
\int_{I_m}(\L^{r_m}_m z,V)_H\dif t=(z,V(t^+_{m-1}))_H\qquad\forall V\in\mathcal{P}^{r_m}(I_m;X),
\end{equation}
 cf.~\cite[Section~4.1]{ScWi10}. Then, in view of this definition with~$X=H_m$ and proceeding as in~\cite{JanssenWihler:15}, we obtain the strong formulation of the dG time stepping method \eqref{eq:dG} on ${I}_m$, viz.,
\begin{equation}\label{eq:dGstrong}
U'(t)+(\L^{r_m}_m\pi_m{\jump{U}_{m-1}})(t)= \Pi^{r_m}_m\F(t,{U}(t)),
\end{equation}
where~$\Pi^{r_m}_m$ denotes the $L^2$-projection onto~$\mathcal{P}^{r_m}(I_m;H_m)$.


\section{A Posteriori Error Analysis}

The goal of this section is to derive $L^\infty$ {\emph{a posteriori}} error bounds for the $hp$-cG and $hp$-dG approximations to~\eqref{eq:1}. {To this end}, we require some structural assumptions on the nonlinearity~$\F$. Specifically, $\F:[0,T]\times H\to H$ is assumed to satisfy $\F(\cdot,0) \in L^1((0,T);H)$ along with the \emph{local} $H$-Lipschitz estimate
\begin{equation}\label{eq:Lip}
\|\F(t,v)-\F(t,w)\|_H \leq  \mathcal{L}(t,\|v\|_H,\|w\|_H)\|v-w\|_H \qquad \forall t \in [0,T] \quad \forall v,w \in H.
\end{equation}
Here, $\mathcal{L}:[0,T]\times \mathbb{R}^+_0 \times \mathbb{R}^{+}_0 \to  \mathbb{R}^{+}_0$ is a \emph{known} function that satisfies $\mathcal{L}(\cdot,a,b) \in L^1((0,T);H)$ for any $a,b \in \mathbb{R}^+_0$ and that is continuous and monotone increasing in the second and third arguments. Under these assumptions on $\F$, problem \eqref{eq:1} admits a unique {(local)} solution $u \in C([0,T];H)$.

\subsection{Preliminary Error Bound}

In order to remedy the suboptimal error estimates that would result from using \eqref{eq:cGstrongsuboptimal} and \eqref{eq:dGstrong} directly, we follow the approach proposed in ~\cite{AkrivisMakridakisNochetto:09} by introducing the \emph{reconstruction} $\widehat{U}_m$ of $U |_{I_{m}}$ which is defined over each closed interval $\bar{I}_m$, $1 \leq m \leq M$, by
\begin{equation}
\begin{split}
\label{reconstruction}
\widehat{U}_m(t) :=\pi_m U(t^-_{m-1}) + \int_{t_{m-1}}^t \Pi_m^{r_m} \F (s,U)  \dif s.
\end{split}
\end{equation}
For $t \in I_m$, $1 \leq m \leq M$, we define the \emph{error} $e(t) := u(t) - U(t) \in H$ where~$U$ is either the $hp$-cG solution from~\eqref{eq:cG} or the $hp$-dG solution from~\eqref{eq:dG}. Since we will be dealing with the reconstruction $\widehat{U}_m$, it will also be necessary to introduce the \emph{reconstruction error} given by $\widehat{e}_m(t) := u(t) -\widehat{U}_m(t) \in H$, $t \in \bar{I}_m$, $1 \leq m \leq M$. We will proceed with the error analysis by first proving an $L^{\infty}$-error bound for $\widehat{e}_m$.

To formulate the error equation, we begin  by substituting \eqref{eq:1} into the definition of $\widehat{e}_m$, viz.,
\begin{equation}
\begin{split}
\widehat{e}_m(t) = u(t_{m-1}) - \widehat{U}_m(t) +  \int_{t_{m-1}}^t \F(s,u) \dif s.
\end{split}
\end{equation}
Adding and subtracting additional terms yields
\begin{equation}
\begin{split}
\widehat{e}_m(t) = \widehat{e}_m(t_{m-1}) + \mathcal{R}(t) +  \int_{t_{m-1}}^t \F(s,u) -\F(s,\widehat{U}_m) \dif s,
\end{split}
\end{equation}
where $\mathcal{R}$ denotes the \emph{residual} given by
\begin{equation}
\begin{split}
\mathcal{R}(t) := \widehat{U}_{m}(t_{m-1})- \widehat{U}_m(t)+  \int_{t_{m-1}}^t \F(s,\widehat{U}_m) \dif s, \qquad t \in {I}_m.
\end{split}
\end{equation}
Using the triangle inequality and Bochner's Theorem implies that 
\begin{equation}
\begin{split}
\|{\widehat{e}}_m(t)\|_H \leq \|\widehat{e}_m(t_{m-1})\|_H+ \|\mathcal{R}(t)\|_H+\int_{t_{m-1}}^t \|\F(s,u) - \F(s,\widehat{U}_m)\|_H \dif s.
\end{split}
\end{equation}
Moreover, applying the local $H$-Lipschitz estimate~\eqref{eq:Lip} together with the monotonicity of $\mathcal{L}$ yields
\begin{equation}
\begin{split}
\label{erroreqn1}
\|{\widehat{e}}_m(t)\|_H \leq \|\widehat{e}_m(t_{m-1})\|_H+ \eta^{\res}_m+\int_{t_{m-1}}^t \mathcal{L}(s,\|\widehat{e}_m\|_H+\|\widehat{U}_m\|_H,\|\widehat{U}_m\|_H)\|\widehat{e}_m\|_H \dif s,
\end{split}
\end{equation}
for $t \in \bar{I}_m$. Here, $\displaystyle \eta^{\res}_m$ denotes the \emph{residual estimator} given by $\displaystyle \eta^{\res}_m := \|\mathcal{R}\|_{L^{\infty}(I_m;H)}$.

On the first interval, {recalling that~$U(t_0^-)=u_0=u(t_0)$}, we can estimate $\|\widehat{e}_m(t_{m-1})\|_H$ directly {by} $\|\widehat{e}_1(t_0)\|_H = \eta_0^{\proj}$ where the \emph{projection estimator} $\eta^{\proj}_m$ is given by 
\begin{equation}\label{eq:etaproj}
\eta^{\proj}_m := \|U(t_m^-) - \pi_{m+1}U(t_{m}^-)\|_H, \qquad m\ge 0. 
\end{equation}
For later intervals, the unknown term $\|\widehat{e}_m(t_{m-1})\|_H$ needs to be replaced with the known term $\|\widehat{e}_{m-1}(t_{m-1})\|_H$. To this end, for $m \geq 1$, we have
\begin{equation}
\begin{split}
\label{reconstructiondiff}
\widehat{e}_{m+1}(t_{m}) - \widehat{e}_{m}(t_{m}) & = \widehat{U}_{m}(t_{m}) - \widehat{U}_{m+1}(t_{m}) \\
& = \pi_{m}U(t_{m-1}^-) - \pi_{m+1}U(t_{m}^-) + \int_{I_m} \Pi_{m}^{r_{m}} \F (s,U)  \dif s.
\end{split}
\end{equation}
Recall that the $hp$-cG method \eqref{eq:cG} satisfies the strong form \eqref{eq:cGstrongsuboptimal} on $I_m$. Thus, we have that
\begin{equation}
\begin{split}
{U}(t^-_{m}) =\pi_{m} U(t^-_{m-1}) + \int_{I_m} \Pi_{m}^{r_{m}-1} \F (s,U)  \dif s.
\end{split}
\end{equation}
By the definition of the $L^2$-projection an equivalent formulation of the above is
\begin{equation}
\begin{split}
{U}(t^-_{m}) =\pi_{m} U(t^-_{m-1}) + \int_{I_m} \Pi_{m}^{r_{m}} \F (s,U)  \dif s.
\end{split}
\end{equation}
Substituting this into \eqref{reconstructiondiff} implies that for the $hp$-cG method we have
\begin{equation}
\begin{split}
\label{cGdiff}
\widehat{e}_{m+1}(t_{m}) - \widehat{e}_{m}(t_{m}) = U(t_m^-) - \pi_{m+1}U(t_{m}^-).
\end{split}
\end{equation}
For the $hp$-dG method \eqref{eq:dG}, we have the strong form \eqref{eq:dGstrong} on $I_m$. Thus, it follows that
\begin{equation}
\begin{split}
{U}(t^-_{m}) =U(t^+_{m-1}) + \int_{I_m} \Pi_{m}^{r_{m}} \F (s,U)  \dif s - \int_{I_m} (\L^{r_m}_m\pi_m{\jump{U}_{m-1}})(s) \dif s.
\end{split}
\end{equation}
{By definition of the lifting operator~$\L^{r_m}_m$ from~\eqref{eq:liftingop} we arrive at}
\begin{equation}
\begin{split}
{U}(t^-_{m}) =U(t^+_{m-1})- \pi_m{\jump{U}_{m-1}} + \int_{I_m} \Pi_{m}^{r_{m}} \F (s,U)  \dif s.
\end{split}
\end{equation}
Equivalently,
\begin{equation}
\begin{split}
{U}(t^-_{m}) = \pi_m U(t^-_{m-1}) + \int_{I_m} \Pi_{m}^{r_{m}} \F (s,U)  \dif s.
\end{split}
\end{equation}
\noindent Since this is the same form as for the $hp$-cG method then for the $hp$-dG method we also have that
\begin{equation}
\begin{split}
\label{dGdiff}
\widehat{e}_{m+1}(t_{m}) - \widehat{e}_{m}(t_{m}) = U(t_m^-) - \pi_{m+1}U(t_{m}^-).
\end{split}
\end{equation}
Applying the triangle inequality to \eqref{cGdiff} and \eqref{dGdiff} yields
\begin{equation}\label{eq:split}
\begin{split}
\|\widehat{e}_{m+1}(t_{m})\|_H \leq \|\widehat{e}_{m}(t_{m})\|_H + \eta^{\proj}_{m},
\end{split}
\end{equation}
{with~$\eta^{\proj}_m$ from~\eqref{eq:etaproj}.} Substituting this result into \eqref{erroreqn1} gives
\begin{equation}
\begin{split}
\label{erroreqn2}
\|{\widehat{e}}_m(t)\|_H \leq \psi_m+\int_{t_{m-1}}^t \mathcal{L}(s,\|\widehat{e}_m\|_H+\|\widehat{U}_m\|_H,\|\widehat{U}_m\|_H)\|\widehat{e}_m\|_H \dif s,
\end{split}
\end{equation}
where~$\psi_m$ is chosen such that
\begin{equation}\label{eq:psi}
\psi_m  \ge \begin{cases}\displaystyle \eta^{\proj}_{m-1}+ \eta^{\res}_m \qquad & \text{if } m = 1 \\
\displaystyle \|\widehat{e}_{m-1}(t_{m-1})\|_H + \eta^{\proj}_{m-1}+ \eta^{\res}_m \qquad & \text{if } m \neq 1 \end{cases}.
\end{equation}
Finally, applying Gronwall's inequality to \eqref{erroreqn2} for $t \in \bar{I}_m$, $1 \leq m \leq M$, yields the following result.

\begin{proposition}
For the cG and dG time stepping schemes~\eqref{eq:cG} and~\eqref{eq:dG}, respectively, there holds the  error bound
\begin{equation}
\begin{split}
\label{eq:gronwall}
\|\widehat{e}_m(t)\|_H & \leq G_m(t)\psi_m,\qquad t \in \bar{I}_m,\quad 1 \leq m \leq M,
\end{split}
\end{equation}
where ~$\psi_m$ satisfies~\eqref{eq:psi} and $G_m$ is given by
\begin{equation}
\begin{split}
G_m(t)  := \exp \bigg (\int_{t_{m-1}}^{t} \mathcal{L}(s,\|\widehat{e}_m\|_H+\|\widehat{U}_m\|_H,\|\widehat{U}_m\|_H) \dif s \bigg).
\end{split}
\end{equation}
\end{proposition}

\subsection{Continuation Argument}

In order to transform~\eqref{eq:gronwall} into an \emph{a posteriori} error bound, we will employ a continuation argument, cf., e.g.,~\cite{CGKM15}. To this end, for $1\le m\le M$, we define the set
\begin{equation}\label{eq:calIm}
\mathcal{I}_m:= \{t\in \bar{I}_m:\,\|\widehat{e}_m\|_{C([t_{m-1},t];H)}\le \delta\psi_m\},
\end{equation}
where $\delta > 1$ is a parameter to be chosen. Note that $\mathcal{I}_m$ is closed since $\widehat{e}_m$ is time-continuous and, obviously, bounded. Moreover, $\mathcal{I}_m$ is assumed to be non-empty; this is certainly true for the first interval since $0 \in \mathcal{I}_1$ and is \emph{a posteriori} verifiable for later intervals. To state the full error bound, we require some additional definitions. Specifically, define the function $\phi_m:[1,\infty) \rightarrow \mathbb{R}$ by 
\begin{equation}\label{eq:phi}
\phi_m(\delta):=\exp\left(\int_{I_m} \mathcal{L}(s,\delta \psi_m+\|\widehat{U}_m\|_H,\|\widehat{U}_m\|_H) \dif s\right)-\delta.
\end{equation}
Additionally, if it exists, we introduce 
\begin{equation}\label{eq:deltam}
\delta_m := \inf \{\delta > 1 : \phi_m(\delta) < 0\} \in [1, \infty),
\end{equation} 
for $m \geq 1$ and we let $\delta_0 := 1$ for convenience.

We are now ready to establish the {following} \emph{conditional a posteriori} error bound for both Galerkin time stepping methods. 
\begin{theorem}\label{maintheorem}
For any $m \geq 1$, if $\delta_1, \ldots, \delta_{m}$ exist then the $hp$-cG scheme~\eqref{eq:cG} and the $hp$-dG scheme~\eqref{eq:dG} satisfy the \emph{a posteriori} error bound
\begin{equation}
\label{eq:prelimerr}
\|\widehat{e}_m\|_{L^\infty(I_m;H)}\le\delta_m\psi_m,
\end{equation}
on $I_m$.
\end{theorem}
\begin{proof}
We omit the trivial case~$\psi_m=0$. Since $\delta_m$ exists, there is some $\delta > 1$ with $\phi_m(\delta)<0$. Suppose that $\widehat{e}_m(t_{m-1})$ exists then $t_{m-1} \in \mathcal{I}_m$ so $\mathcal{I}_m$ is non-empty, closed and bounded and thus has some maximal time $t^{\star}$. Let us first make the assumption that $t^{\star} < t_m$ and work towards a contradiction. Substituting the error bound from $\mathcal{I}_m$ into \eqref{eq:gronwall} and using the monotonicity of $\mathcal{L}$ implies that 
\begin{equation}
\begin{split}
\notag
\|\widehat{e}_m\|_{C([t_{m-1},t^{\star}];H)} & \leq  \exp \bigg (\int_{I_m} \mathcal{L}(s,\delta\psi_m+\|\widehat{U}_m\|_H,\|\widehat{U}_m\|_H) \dif s \bigg)\psi_m=\left(\phi_m(\delta)+\delta\right)\psi_m<\delta\psi_m.
\end{split}
\end{equation}
This is a contradiction since we had assumed that the set $\mathcal{I}_m$ had maximal time $t^{\star}<t_m$. Hence, it follows that $\|\widehat{e}_m\|_{L^\infty(I_m;H)} \le \delta\psi_m$. Taking the limit~$\delta\to\delta_m$ we deduce~\eqref{eq:prelimerr}. To complete the proof, we note that we can conclude by recursion that $\widehat{e}_m(t_{m-1})$ exists if $\delta_0, \ldots, \delta_{m-1}$ exist and $0 \in \mathcal{I}_1$. Since $0 \in \mathcal{I}_1$ trivially and $\delta_0, \ldots, \delta_{m-1}$ exist by premise, the original assumption that $\widehat{e}_m(t_{m-1})$ exists is unneeded.  
\end{proof}

An important question arises with regard to Theorem \ref{maintheorem} {\bf --} is it possible that $\delta_m$ \emph{never} exists for certain nonlinearities $\F$? The following lemma provides an answer to this question. 

\begin{lemma}\label{lm:ex}
For any $m \geq 1$, if $\delta_0, \ldots, \delta_{m-1}$ exist and the time step length~$k_m>0$ is chosen sufficiently small then the set~$\{\delta>1:\,\phi_m(\delta)<0\}$ is non-empty, $\delta_m\ge1$ from~\eqref{eq:deltam} exists and~$\phi_m(\delta_m)=0$.
\end{lemma}
\begin{proof}
Since $\delta_0, \ldots, \delta_{m-1}$ exist, Theorem \ref{maintheorem} implies that $\psi_m$ exists and thus $\phi_m$ is well-defined. For fixed $\delta^\star>1$ and upon setting~$\epsilon:=\nicefrac12(\delta^\star-1)>0$, we can choose~$k_m>0$ small enough so that
\[
\exp\left(\int_{I_m} \mathcal{L}(s,\delta^{\star} \psi_m+\|\widehat{U}_m\|_H,\|\widehat{U}_m\|_H) \dif s\right)<1+\epsilon.
\]
A quick calculation reveals that $\phi_m(\delta^\star)<1+\epsilon-\delta^\star=\nicefrac12(1-\delta^\star)<0$. Therefore, for~$k_m>0$ sufficiently small, the set~$\{\delta>1:\,\phi_m(\delta)<0\}$ is non-empty. Furthermore, since $\phi_m$ is continuous and $\phi_m(1) \geq 0$, it follows that~$\delta_m$ exists and satisfies~$\phi_m(\delta_m) = 0$. 
\end{proof}

In some sense, $\widehat{U}_m$ is a better approximation to $u |_{I_m}$ than $U |_{I_m}$; thus from a practical standpoint it is often best to use Theorem \ref{maintheorem} directly, however, for some applications a bound on the error rather than on the reconstruction error may be required so we introduce the following corollary.

\begin{corollary}\label{maintheoremcorollary}
For any $m \geq 1$, if $\delta_1, \ldots, \delta_{m}$ exist then the $hp$-cG scheme~\eqref{eq:cG} and the $hp$-dG scheme~\eqref{eq:dG} satisfiy the \emph{a posteriori} error bound
\begin{equation}
\|{e}\|_{L^\infty(I_m;H)}\le\delta_m\psi_m + \|U - \widehat{U}_m\|_{L^{\infty}(I_m;H)},
\end{equation}
on $I_m$.
\end{corollary}
\begin{proof}
The bound follows directly from rewriting the error, viz., $e = \widehat{e}_m + \widehat{U}_m - U$, the triangle inequality and the reconstruction error bound of Theorem \ref{maintheorem}.
\end{proof}

\subsection{Computable Error Bound}
In order to yield fully computable error bounds we must give an explicit characterization of $\psi_m$ from~\eqref{eq:psi}. Theorem \ref{maintheorem} implies that
\begin{equation}
\|\widehat{e}_{m-1}(t_{m-1})\|_H \leq \|\widehat{e}_{m-1}\|_{L^{\infty}(I_{m-1};H)} \leq \delta_{m-1}\psi_{m-1}.
\end{equation}
Thus, we can define~$\psi_m$ by 
\begin{equation}\label{psi-def}
\psi_m  := \begin{cases}\displaystyle \eta^{\proj}_{m-1}+ \eta^{\res}_m \qquad & \text{if } m = 1 \\
\displaystyle \delta_{m-1}\psi_{m-1}+ \eta^{\proj}_{m-1}+ \eta^{\res}_m \qquad & \text{if } m \neq 1 \end{cases}.
\end{equation}
Defining $\psi_m$ in this way yields a \emph{recursive error estimator} and makes the error bounds of Theorem \ref{maintheorem} and Corollary \ref{maintheoremcorollary} \emph{fully computable}. 

In order to develop adaptive algorithms that can approximate the blow-up time of a nonlinear initial value problem, it is important to interpret the role that~$\delta_m$ plays in the error bound of Theorem \ref{maintheorem}. Recalling the bound and applying~\eqref{psi-def}, we see that the reconstruction error on $I_m$ satisfies
\begin{equation}
\begin{split}
\|\widehat{e}_m\|_{L^{\infty}(I_m;H)} \leq \delta_m (\delta_{m-1}\psi_{m-1}+ \eta^{\proj}_{m-1}+ \eta^{\res}_m).
\end{split}
\end{equation}
Of the three components that make up the error estimator, the term $\delta_{m-1}\psi_{m-1}$ is a bound for the error on the previous time step while $\displaystyle \eta^{\proj}_{m-1}+ \eta^{\res}_m$ represents the local (additive) contribution to the error estimator on the current time step. The correct interpretation of $\delta_m$, then, is that it is an \emph{a posteriori} approximation to the growth rate of the error on $I_m$; this becomes clear upon consideration of the following simple example. In fact, the following corollary shows that $\delta_m$ is the expected local growth rate for globally $H$-Lipschitz $\F$.

\begin{corollary}\label{HLipschitz}
Suppose that $\F$ is globally $H$-Lipschitz with constant $\mathcal{L} = L \ge 0$ then $\delta_m = \mathrm{e}^{Lk_m}$ and, thus, the error bound of Theorem \ref{maintheorem} holds unconditionally on $I_m$, viz.,
\begin{equation}
\|\widehat{e}_m\|_{L^\infty(I_m;H)}\le \mathrm{e}^{Lk_m}\psi_m.
\end{equation}
\end{corollary}
\begin{proof}
From the definition of the function $\phi_m$, it follows that $\phi_m(\delta) = \mathrm{e}^{L k_m} - \delta$. Therefore, $\displaystyle \delta_m = \inf \{\delta > 1 : \mathrm{e}^{L k_m} - \delta < 0 \} = \mathrm{e}^{L k_m}$. Since $\delta_m$ exists and is finite for any $k_m$, we conclude that the error bound of Theorem \ref{maintheorem} holds unconditionally on $I_m$.
\end{proof}


\section{Adaptive Algorithms}

The error estimators derived in the previous section will form the basis of an adaptive strategy to estimate the blow-up time of \eqref{eq:1}. In particular, we will consider both a $h$-adaptive approach and an $hp$-adaptive approach. For the remainder of this section, we assume that $H_1 = \ldots = H_M$ for simplicity but remark that both adaptive algorithms can be easily modified to account for variable finite-dimensional subspaces.

\subsection{A $h$-Adaptive Approach}

The first difficulty encountered  in the construction of a $h$-adaptive algorithm is that both \eqref{eq:cG} and \eqref{eq:dG} are implicit methods which means that the existence of a numerical solution is not guaranteed. It is tempting to conduct an existence analysis such as in \cite{JanssenWihler:15} to obtain bounds on the length of the time steps needed in order to yield a numerical solution, however, such bounds are inherently pessimistic. Since existence can be determined \emph{a posteriori}, our $h$-adaptive algorithm just reduces the time step length until a numerical solution exists.

The second difficulty is how to approximate $\delta_m$; unfortunately, it is impossible to give a precise characterization for how to do this for any given $\F$ primarily because $\F$ may be `pathological'.  Fortunately, however, most nonlinearities of interest do not fall into this category. Suppose that $\F$ is chosen such that $\phi_m$ has, at most, a small finite number of zeros then it should be possible to approximate the zeros of $\phi_m$ numerically. Since $\delta_m$ satisfies $\phi_m(\delta_m) = 0$, we then only need to check the roots of $\phi_m$ and verify that one of our numerical approximations, $\displaystyle \tilde{\delta}_m$, satisfies ${\displaystyle \phi_m(\tilde{\delta}_m) < 0}$. In our numerical experiments, we employ a Newton method to find $\tilde{\delta}_m$  with an initial guess close to one on $I_1$ (the proof of Lemma \ref{lm:ex} implies that $\delta_m \approx 1$ for $k_m \approx 0$) and an initial guess close to ${\delta}_{m-1}$ on $I_{m}$ for $m \geq 2$; this approach works well for the studied problems.
   
As is standard in finite element algorithms for linear problems, the time step is halved and the numerical solution recomputed until $\displaystyle \eta^{\res}_m$ is below the \emph{tolerance} $\tt tol$, however, we must also account for the nonlinear scaling that enters through $\delta_m$. The structure of the error estimator implies that the interval $I_1$ is the most important since the residual estimator on $I_1$ propagates through the remainder of the error estimator. Similarly, each successive subinterval $I_m$ is less important than the previous subinterval $I_{m-1}$ with the term $\delta_{m-1}$ measuring the extent to which this is true. To account for this, we increase the tolerance by the factor $\delta_m$ after computations on $I_m$ are complete.

We then advance to the next interval using the previous time step length as a reference and continue in this way until $\delta_m$ no longer exists; the adaptive algorithm is then terminated and it outputs the total number of degrees of freedom (DoFs) used along with the sum of all the time step lengths, $T$, as an estimate for $T_{\infty}$. The pseudocode for the $h$-adaptive algorithm is given in Algorithm 1.

\begin{algorithm} \label{algorithm1}
  \begin{algorithmic}[1]
     \State {\bf Input:} $\F$, $u_0$, $r$, $k_1$, ${\tt tol}$.
               \State Set~$r_1= r$, attempt to compute $U |_{I_1}$.
     \While {$U |_{I_1}$ does not exist}
          \State $k_1 \leftarrow \nicefrac{k_1}{2} $, and attempt to compute $U |_{I_1}$.
\EndWhile

\State Compute $\eta^{\res}_1$.

\While {$\eta^{\res}_1 > {\tt tol}$}
          \State $\displaystyle k_1 \leftarrow \nicefrac{k_1}{2}$,
          compute $U |_{I_1}$
          and determine $\eta^{\res}_1$.
\EndWhile

\State Set $m = 0$ and attempt to compute ${\delta}_{1}$.
     \While {${\delta}_{m+1}$ exists}
     \State $m \leftarrow m+1$,
     ${\tt tol} \leftarrow {\delta}_m*{\tt tol}$,
     $r_{m+1} = r_m$,
     $k_{m+1} = k_m$.
              \State Attempt to compute $U |_{I_{m+1}}$.
     \While {$U |_{I_{m+1}}$ does not exist}
          \State $k_{m+1} \leftarrow \nicefrac{k_{m+1}}{2}$,
          and attempt to compute $U |_{I_{m+1}}$.
\EndWhile

\State Compute $\eta^{\res}_{m+1}$.

\While {$\eta^{\res}_{m+1} > {\tt tol}$}
          \State $\displaystyle k_{m+1} \leftarrow \nicefrac{k_{m+1}}{2}$,
          compute $U |_{I_{m+1}}$
          and determine $\eta^{\res}_{m+1}$.
\EndWhile

\State Attempt to compute ${\delta}_{m+1}$.
\EndWhile
\State {\bf Output:} $m$, $t_{m}$.
  \end{algorithmic}
  \caption{$h$-adaptive time stepping algorithm (fixed polynomial degree~$r$)}
\end{algorithm}

\subsection{An $hp$-Adaptive Approach}

The basic outline of the $hp$-adaptive algorithm will be the same as that of the $h$-adaptive algorithm; the only difficulty lies in how we choose locally between $h$-refinement and $p$-refinement. The theory of the $hp$-FEM suggests that $p$-refinement is superior to $h$-refinement if the solution is `smooth enough', so we perform $p$-refinement if $U |_{I_m}$ is smooth; otherwise, we do $h$-refinement. The pseudocode for the $hp$-adaptive algorithm is very similar to Algorithm~1; the difference lies in replacing the simple time step bisections in lines~(8:) and~(19:) by the following procedure:
\begin{algorithmic}[]
\If {$U|_{I_{m+1}}$ is smooth}
\State $r_{m+1} \leftarrow r_{m+1}+1$.
          \Else
          \State $\displaystyle k_{m+1} \leftarrow \nicefrac{k_{m+1}}{2}$.
          \EndIf
          \State Compute $U |_{I_{m+1}}$ and determine $\eta^{\res}_{m+1}$.
\end{algorithmic}

In the numerical experiments, we consider only real-valued ODEs, and so we remark specifically on the estimation of smoothness in this special case. There are many ways to estimate smoothness of the numerical solution (see \cite{MM14} for an overview); we choose to use a computationally simple approach from \cite{FWW14} based on Sobolev embeddings. Here, the basic idea is to monitor the constant in the Sobolev embedding~$H^1(I_m)\hookrightarrow L^\infty(I_m)$ in order the classify a given function as  locally smooth or not. Specifically, we define the smoothness indicator~$\theta_m:\,H^1(I_m)\to[0,1]$ by
\begin{equation}
\begin{split}
\theta_m[u] := \begin{cases}
\displaystyle \|u\|_{L^{\infty}(I_m)} \left(k^{-\nicefrac{1}{2}}_m\|u\|_{L^2(I_m)} + \frac{1}{\sqrt{2}}k_m^{\nicefrac{1}{2}}\|u'\|_{L^2(I_m)} \right)^{-1} \qquad & \text{if } u \not\equiv 0,\\[2ex]
\displaystyle 1& \text{if } u\equiv 0,
\end{cases}
\end{split}
\end{equation}
which, intuitively, takes values close to one if~$u$ is smooth and values close to zero otherwise; see the reasoning of \cite{FWW14} for details. Following~\cite{FWW14}, we characterize $U|_{I_m} \in \mathcal{P}^{r_m}(I_m;\mathbb{R})$ as smooth if
\begin{equation}
\begin{split}
\theta_m \left[ \frac{\dif^{r_m-1} U}{\dif t^{r_m-1}} \right] \geq \theta^{\star};
\end{split}
\end{equation}
here, values around $\theta^{\star} = 0.85$ were observed to produce the best results in our numerical experiments below.

We conclude this section with a corollary on the magnitude of the reconstruction error under either the $h$-adaptive algorithm or the $hp$-adaptive algorithm. In order to state the corollary, we require some additional notation. To that end, we denote the \emph{initial tolerance} by $\tt tol^*$ and define the \emph{a posteriori} approximation to the growth rate of the error on $(0,t_m)$ by 
\[
\widehat{\delta}_m := \prod_{i=0}^m \delta_i.
\]

We are now ready to state the corollary.
\begin{corollary}
\label{algorithmbound}
Suppose that $H_1=\ldots=H_M$ and that $\delta_1, \ldots, \delta_M$ exist then under the $h$-adaptive algorithm or the $hp$-adaptive algorithm, the reconstruction error satisfies
\begin{equation}
\begin{split}
\notag
\max_{1 \leq m \leq M} \|\widehat{e}_m\|_{L^{\infty}(I_m;H)} \leq M \widehat{\delta}_M \tt \tol^\star.
\end{split}
\end{equation}
\end{corollary}
\begin{proof}
To begin the proof, we will first prove by induction that $\psi_m \leq m\widehat{\delta}_{m-1} \tt tol^\star$. For $I_1$, we have that $\psi_1 = \eta^{\res}_1 \leq \tt tol^\star$, so the base case is verified. Assuming that the bound is true for $m-1$, then recalling the definition of $\psi_m$ from \eqref{psi-def} as well as the scaling nature of the tolerances in the $h$-adaptive and $hp$-adaptive algorithms, that is, $\eta^{\res}_m\le\widehat\delta_{m-1}{\tt tol}^\star$, we have
\begin{equation}
\begin{split}
\notag
\psi_m & = \delta_{m-1}\psi_{m-1}+ \eta^{\res}_m \\
&\leq (m-1)\delta_{m-1}\widehat{\delta}_{m-2}{\tt tol}^\star + \widehat{\delta}_{m-1}{\tt tol}^\star \\
&= (m-1)\widehat{\delta}_{m-1}{\tt tol}^\star+\widehat{\delta}_{m-1}\tt tol^\star \\
&= m\widehat{\delta}_{m-1}\tt tol^\star.
\end{split}
\end{equation}
Thus the stated bound holds for any $1 \leq m \leq M$. \mbox{Theorem \ref{maintheorem} combined with this bound yields}
\begin{equation}
\begin{split}
\notag
\max_{1 \leq m \leq M} \|\widehat{e}_m\|_{L^{\infty}(I_m;H)} \leq \max_{1 \leq m \leq M} \delta_m \psi_m \leq \max_{1 \leq m \leq M} m \widehat{\delta}_{m} {\tt tol}^\star = M\widehat{\delta}_M {\tt tol}^\star.
\end{split}
\end{equation}
This completes the proof.
\end{proof}


\section{Numerical Experiments}\label{numex}

In this section, we will apply the adaptive algorithms developed in the previous section to two real-valued initial value problems. In both numerical experiments, we approximate the implicit Galerkin methods \eqref{eq:cG} and \eqref{eq:dG} with an explicit Picard-type iteration; cf.~\cite{JanssenWihler:15}.
 
\subsection{Example 1}
 
In this numerical example, we consider \eqref{eq:1} with the polynomial nonlinearity $\F(t,u) = u^2$. Through separation of variables the exact solution is given by 
\begin{equation}
\begin{split}
\notag
u(t) = \frac{u_0}{1-u_0t},
\end{split}
\end{equation}
which has blow-up time given by $T_{\infty} = u_0^{-1}$. Note that for any $v_1, v_2 \in \mathbb{R}$, the nonlinearity $\F$ satisfies
\begin{equation}
\begin{split}
|\F(v_1)-\F(v_2)| = |v_1^2 - v_2^2 | \leq |v_1 - v_2|(|v_1|+|v_2|),
\end{split}
\end{equation}
so we set $\mathcal{L}(|v_1|,|v_2|) = |v_1|+|v_2|$. Thus, in this case, we have
\begin{equation}
\begin{split}
\phi_m(\delta) = \exp\bigg(k_m\delta \psi_m+2 \int_{I_m} |\widehat{U}_m| \dif s \bigg)-\delta
\end{split}
\end{equation}
in \eqref{eq:phi}.
 
\begin{figure}
\centering
\includegraphics[scale=0.265]{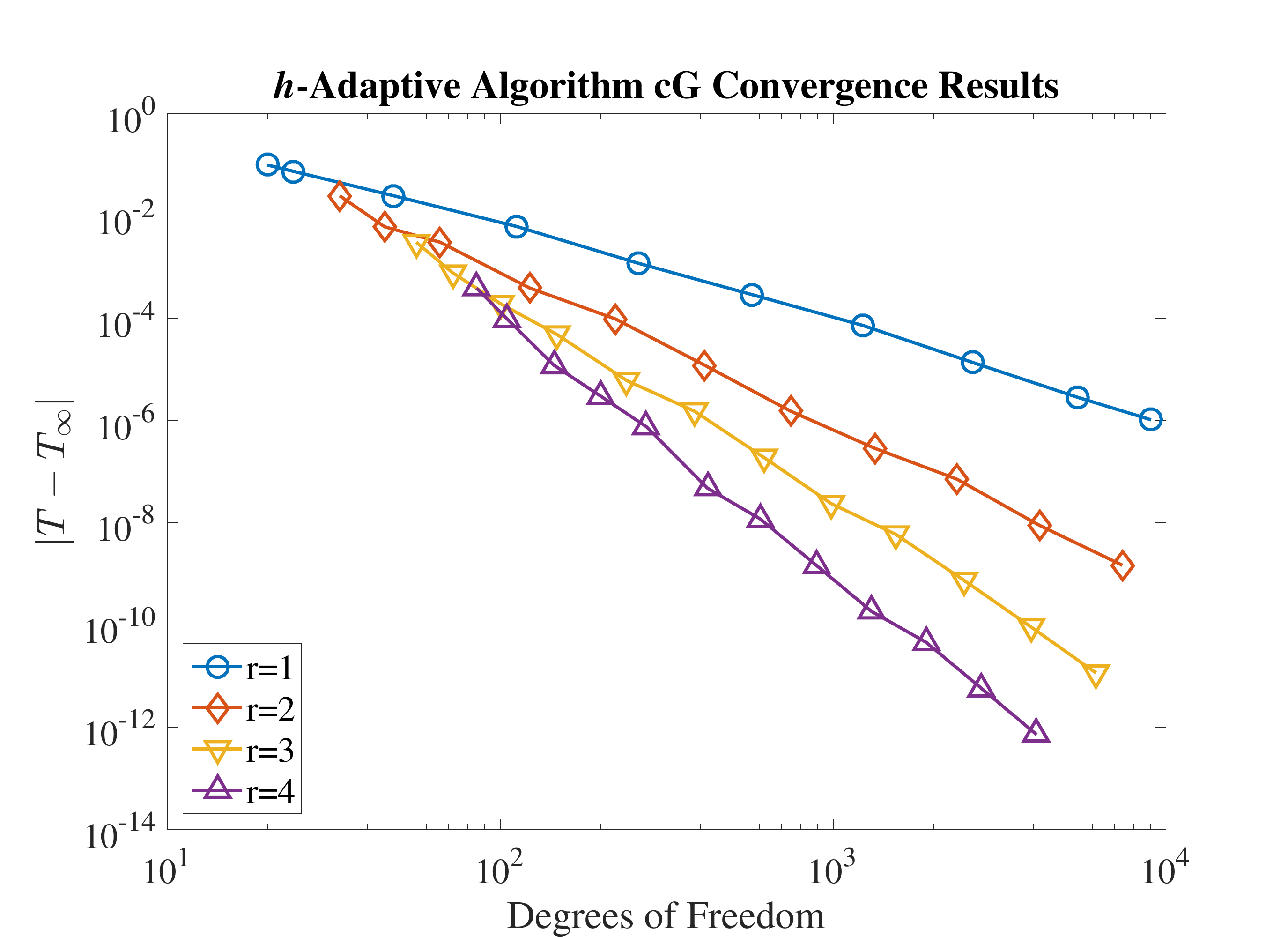} \includegraphics[scale=0.265]{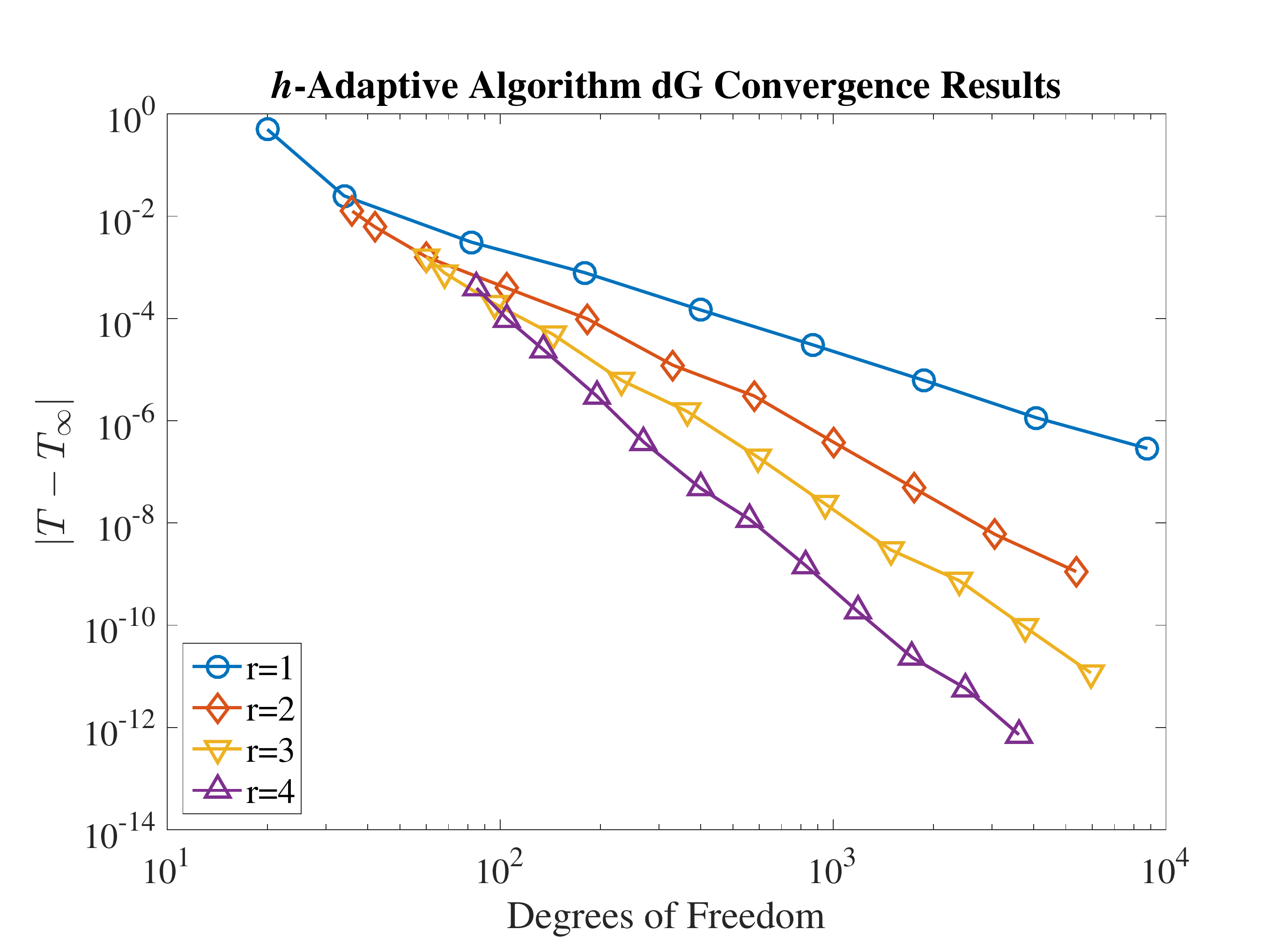}
\includegraphics[scale=0.265]{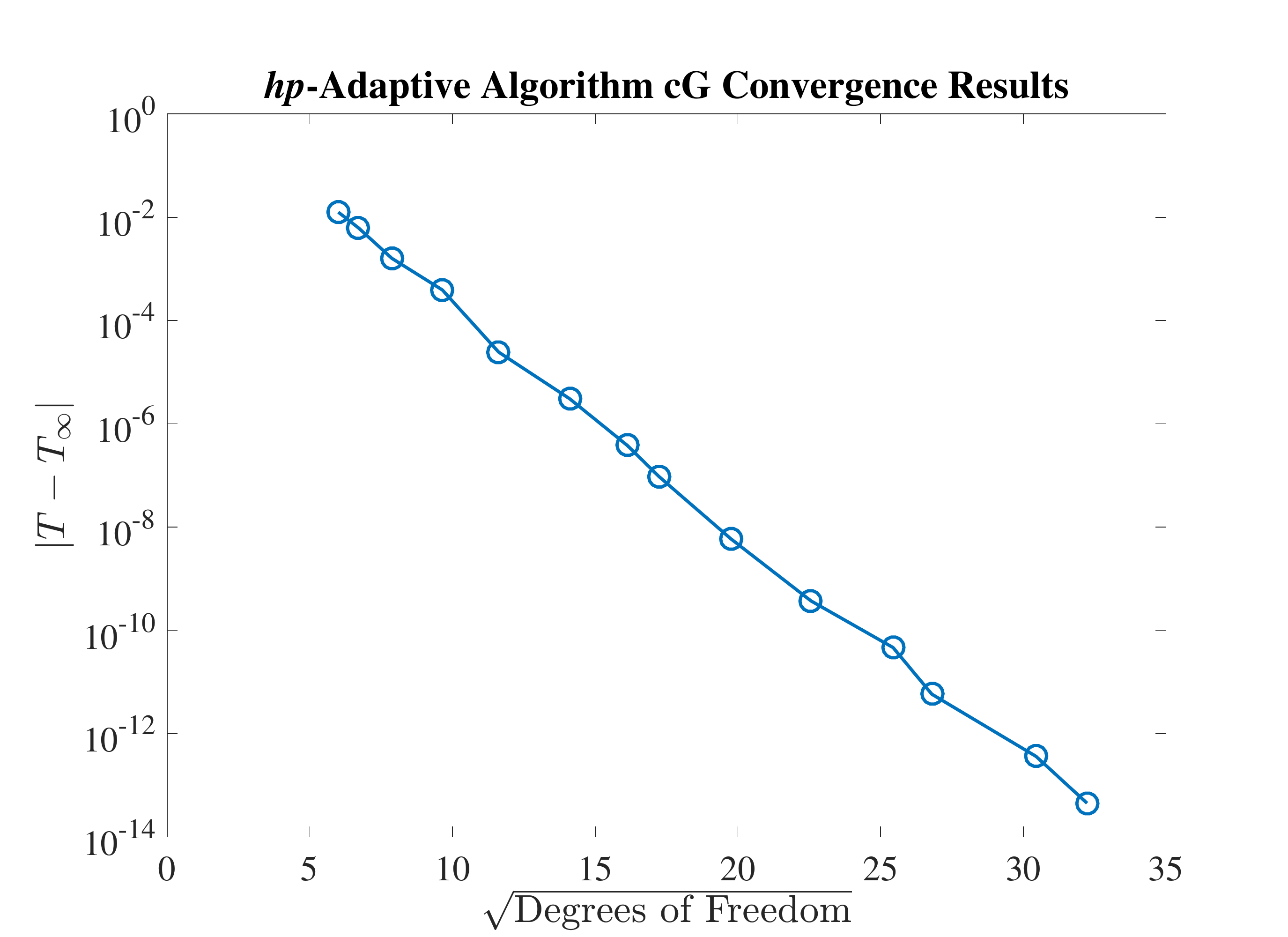} \includegraphics[scale=0.265]{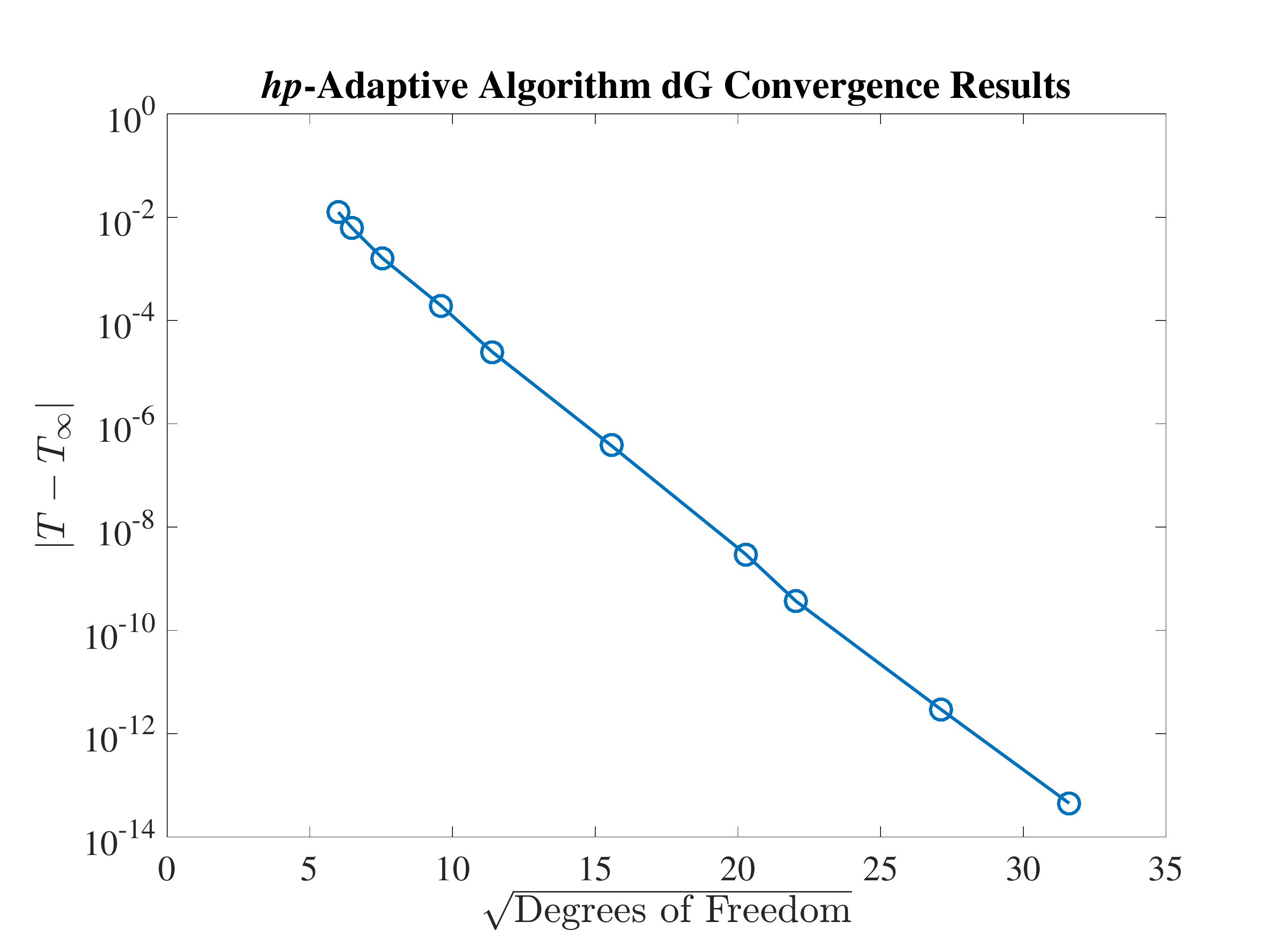}
\caption{Example 1: adaptive algorithm convergence results.}
\label{example1results}
\end{figure}

We begin by applying the $h$-adaptive algorithm to \eqref{eq:1} with initial condition $u(0) = 1$ for both Galerkin time stepping methods utilizing polynomials up to and including degree~{$r=4$}. We reduce the tolerance and observe the rate of convergence of the blow-up time error $|T - T_{\infty}|$. The results given in Figure \ref{example1results} show that
\begin{equation}
\begin{split}
\notag
|T - T_{\infty}| \propto (\text{\#DoFs})^{-(r+1)},
\end{split}
\end{equation}
for both Galerkin methods where $r$ is the polynomial degree and \#DoFs signifies the total number of degrees of freedom. Next, we apply the $hp$-adaptive algorithm to \eqref{eq:1} with the same data. We reduce the tolerance and observe exponential convergence of the blow-up time error for both Galerkin methods, viz.,
\begin{equation}
\begin{split}
\notag
|T - T_{\infty}| \propto \mathrm{e}^{-\sqrt{{b}\text{\#DoFs}}},
\end{split}
\end{equation}
with some constant~$b>0$, as shown in Figure \ref{example1results}.

For the $hp$-adaptive algorithm under the $hp$-cG scheme~\eqref{eq:cG}, we also observe the \emph{effectivity indices} of the error estimator (with respect to the reconstruction error) given by
\begin{equation}
\begin{split}
\notag
\text{effectivity index } |_{I_m} := \frac{\delta_m \psi_m}{\displaystyle \max_{1 \leq k \leq m} \|\widehat{e}_k\|_{L^{\infty}(I_k)} },
\end{split}
\end{equation}
over the course of each computational run for different tolerances and record the best effectivity indices observed in a given computational run versus the inverse of the distance to the blow-up time in Figure \ref{example1deltaEI}. The results show that the error estimator seems to account well for the behaviour of the reconstruction error in certain situations with an effectivity index of 70 for a blow-up problem being particularly impressive; in other situations, however, the effectivity indices observed are much larger. Such a large variation in the effectivity indices observed is actually to be \emph{expected} since in order to provide an upper bound for the error in any situation, the error estimator must account for the \emph{worst possible scenario} to a blow-up problem which may not actually be realised in practise.

Finally, for the $hp$-adaptive algorithm under the $hp$-cG scheme~\eqref{eq:cG}, we plot the value of $\widehat{\delta}_m$ versus the inverse of the distance to the blow-up time over the course of the final computation and display the results in Figure \ref{example1deltaEI}. If we denote the distance from the blow-up time at time $t_m$ by $\varepsilon_m := |t_m - T_{\infty}|$ then Figure \ref{example1deltaEI} implies the relationship
\begin{equation}
\begin{split}\notag
\widehat{\delta}_m \propto \varepsilon^{-2}_m.
\end{split}
\end{equation}
Thus from Corollary \ref{algorithmbound}, we infer that there exists some constant $C>0$ that is independent of the distance to the blow-up time and the maximum time step length such that
\begin{equation}
\begin{split}
\notag
\max_{1 \leq m \leq M} \|\widehat{e}_m\|_{L^{\infty}(I_m)} \leq CM\varepsilon^{-2}_M \tt tol^{\star}.
\end{split}
\end{equation}
Since $\displaystyle \|u\|_{L^\infty(0,T)} = \varepsilon_M^{-1}$, we conclude that the error estimator blows up at a faster rate than the exact solution in this example; moreover, we infer from Figure \ref{example1deltaEI} that in the best case scenario (for the error estimator), this rate appears to be quasi-optimal in the sense that it mirrors the rate of the reconstruction error. Finally, we remark that this result gives weight to the interpretation of $\widehat{\delta}_M$ as an \emph{a posteriori} approximation to the blow-up rate of the error on $(0,T)$ for nonlinear $\F$.

\begin{figure}
\centering
\includegraphics[scale=0.265]{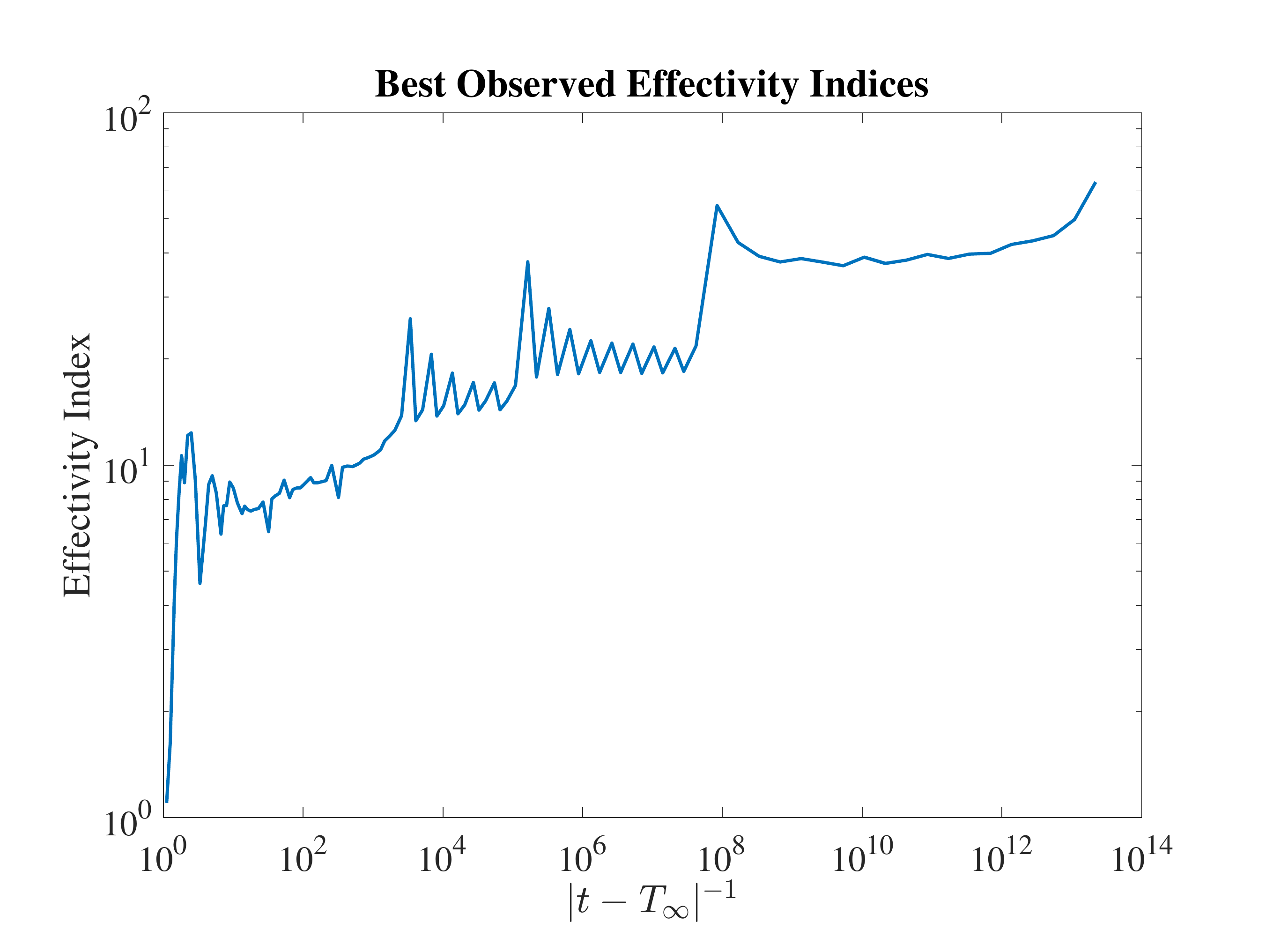} \includegraphics[scale=0.265]{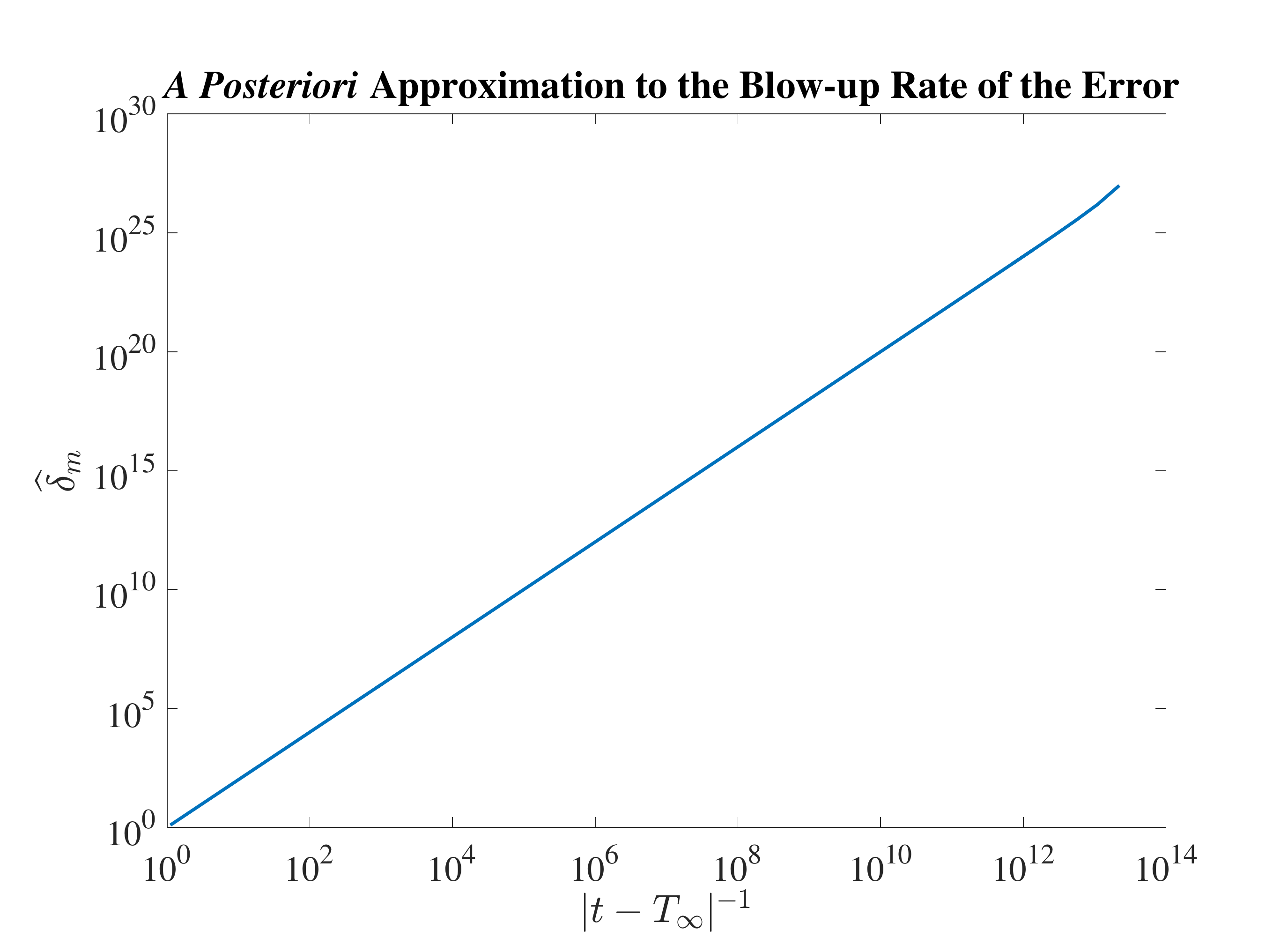}
\caption{Example 1: best observed effectivity indices and $\widehat{\delta}_m$.}
\label{example1deltaEI}
\end{figure}
 
\subsection{Example 2}

In this numerical example, we consider \eqref{eq:1} with the exponential nonlinearity $\F(t,u) = \mathrm{e}^u$. Through separation of variables the exact solution is given by 
\begin{equation}
\begin{split}
\notag
u(t) = \log \bigg(\frac{\mathrm{e}^{u_0}}{1-\mathrm{e}^{u_0}t}\bigg),
\end{split}
\end{equation}
which has blow-up time given by $T_{\infty} = \mathrm{e}^{-u_0}$. Note that for any $v_1, v_2 \in \mathbb{R}$, the nonlinearity $\F$ satisfies
\begin{equation}
\begin{split}
|\F(v_1)-\F(v_2)| = |\mathrm{e}^{v_1} -\mathrm{e}^{v_2}| \leq \frac{1}{2}|v_1 - v_2|(\mathrm{e}^{|v_1|}+\mathrm{e}^{|v_2|}),
\end{split}
\end{equation}
so we set $\mathcal{L}(|v_1|,|v_2|) = \nicefrac{1}{2}(\mathrm{e}^{|v_1|}+\mathrm{e}^{|v_2|})$. Thus for this nonlinearity, we have
\begin{equation}
\begin{split}
\phi_m(\delta) = \exp\bigg(\frac{1}{2}(1+\mathrm{e}^{\delta\psi_m})\int_{I_m} \mathrm{e}^{|\widehat{U}_m|} \dif s \bigg)-\delta
\end{split}
\end{equation} 
in~\eqref{eq:phi}.

\begin{figure}
\centering
\includegraphics[scale=0.265]{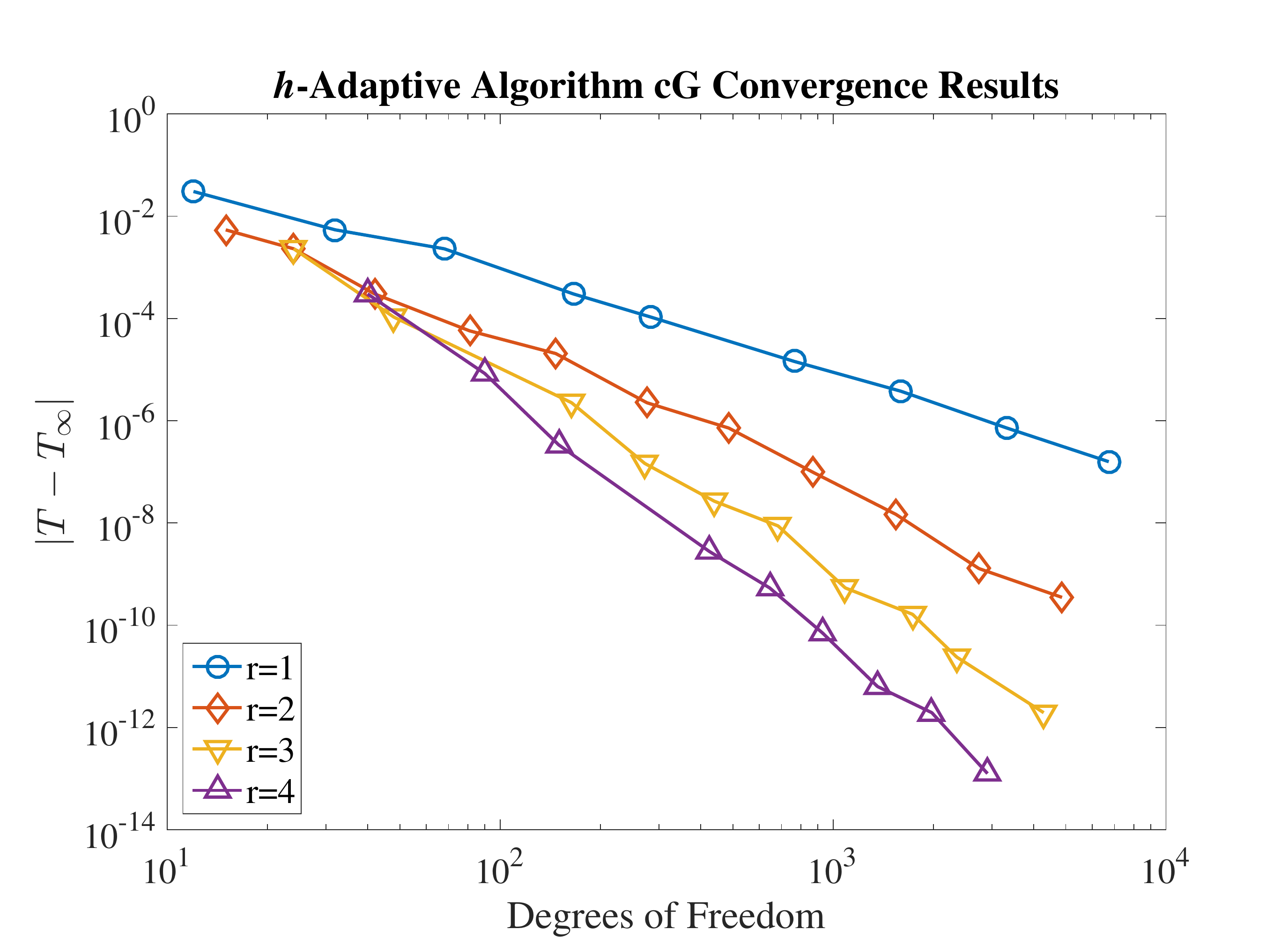} \includegraphics[scale=0.265]{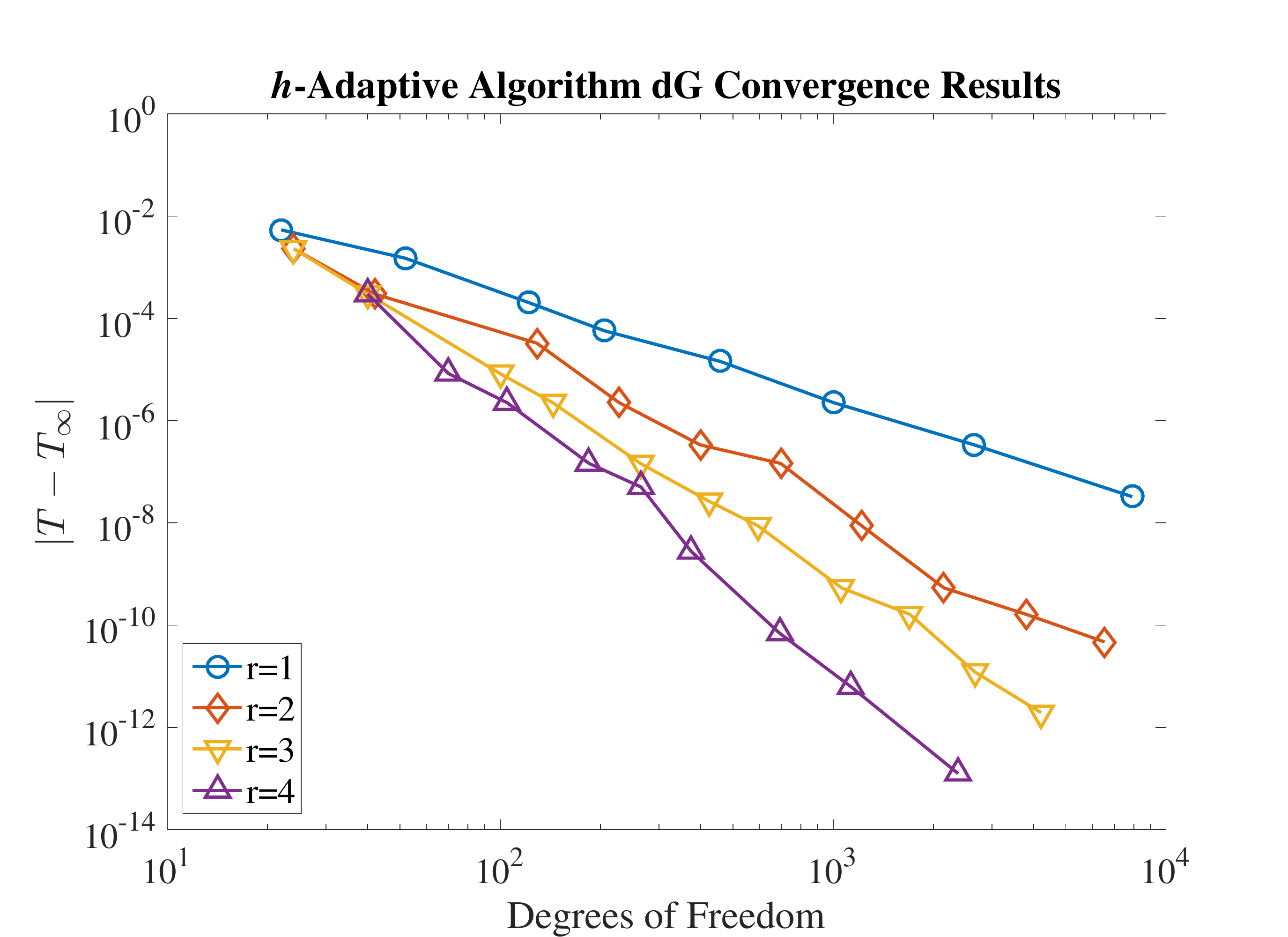}
\includegraphics[scale=0.265]{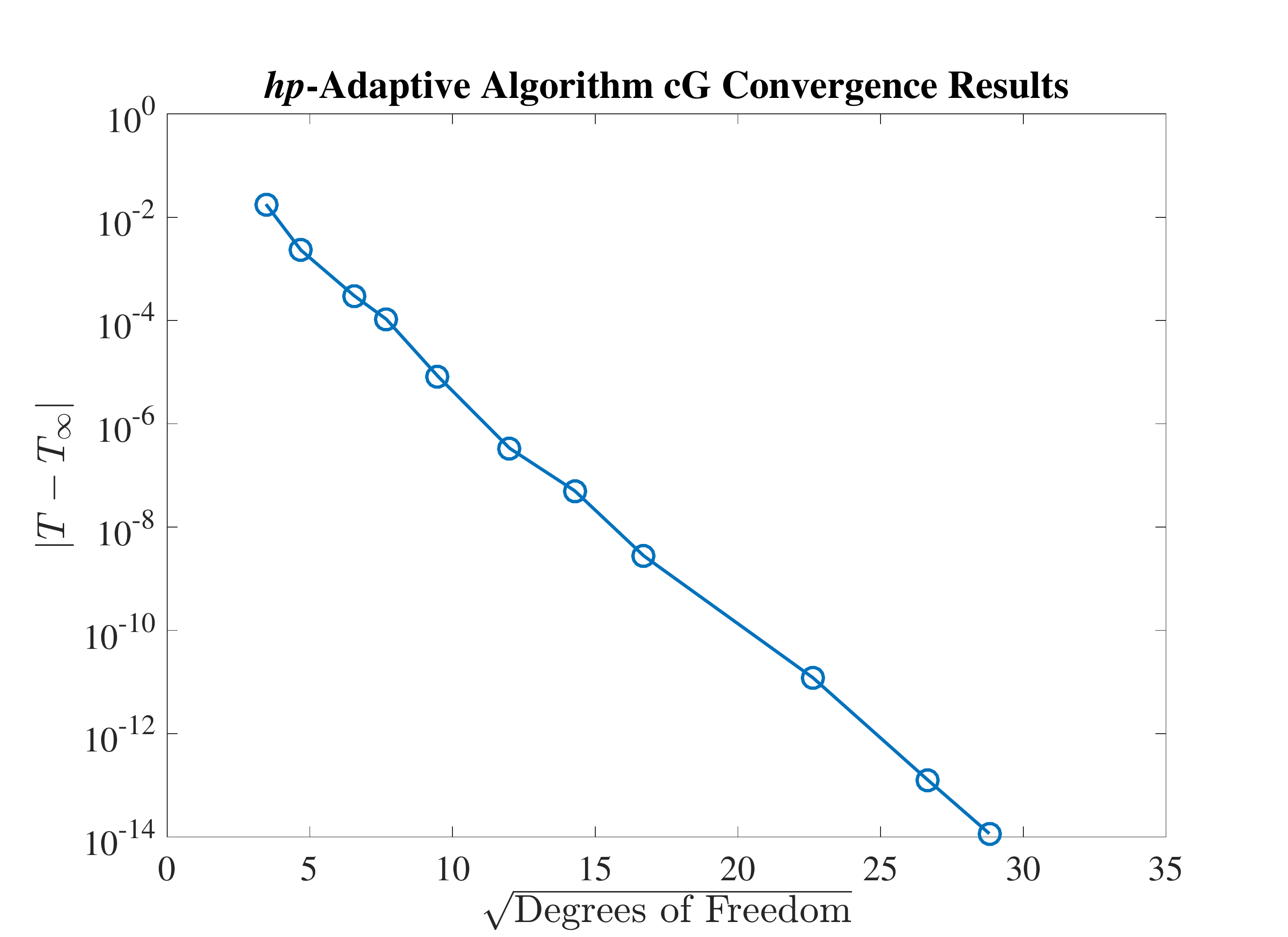} \includegraphics[scale=0.265]{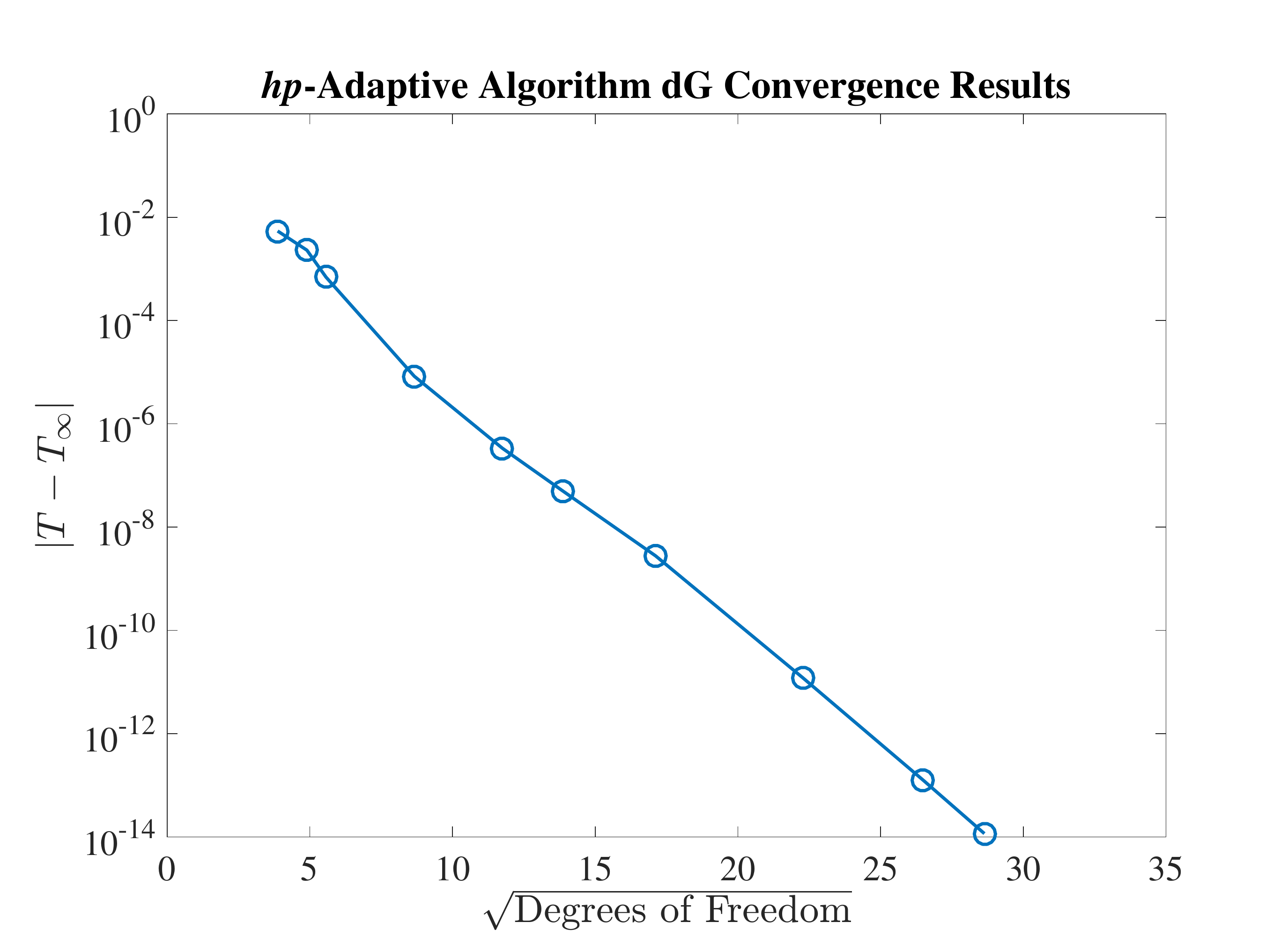}
\caption{Example 2: adaptive algorithm convergence results.}
\label{example2results}
\end{figure}

As in Example 1, we apply the $h$-adaptive algorithm to \eqref{eq:1} with initial condition $u(0) = 1$ for both Galerkin methods utilizing polynomials up to and including degree~$r=4$. We reduce the tolerance and observe the rate of convergence to the blow-up time. The results given in Figure \ref{example2results} show that for this example we also have that
\begin{equation}
\begin{split}
\notag
|T - T_{\infty}| \propto (\text{\#DoFs})^{-(r+1)},
\end{split}
\end{equation}
for both Galerkin methods where $r$ is the polynomial degree. Next, we apply the $hp$-adaptive algorithm to \eqref{eq:1} with the same data. As in Example 1, the tolerance is reduced and the results given in Figure \ref{example2results} show exponential convergence to the blow-up time, viz.,
\begin{equation}
\begin{split}
\notag
|T - T_{\infty}| \propto \mathrm{e}^{-\sqrt{{b}\text{\#DoFs}}},
\end{split}
\end{equation}
with some constant~$b>0$, for both Galerkin methods.

Additionally, for the $hp$-adaptive algorithm under the $hp$-cG scheme~\eqref{eq:cG}, we observe the effectivity indices (with respect to the reconstruction error) over the course of each computational run for different tolerances and record the best effectivity indices observed in a given computational run in Figure \ref{example2deltaEI}. The range of effectivity indices observed over all computational runs is comparable to the range observed in the previous example with large variation and an effectivity index of around 60 in the best case, cf. Figure \ref{example2deltaEI}.

Finally, for the $hp$-adaptive algorithm under the $hp$-cG scheme~\eqref{eq:cG}, we plot the value of $\widehat{\delta}_m$ over the course of the final computation with the results displayed in Figure \ref{example2deltaEI}. With the notation from the previous example, we deduce that
\begin{equation}
\begin{split}\notag
\widehat{\delta}_m \propto \varepsilon^{-1}_m,
\end{split}
\end{equation}
for this example. Thus from Corollary \ref{algorithmbound}, we infer that there exists some constant $C>0$ that is independent of the distance to the blow-up time and the maximum time step length such that
\begin{equation}
\begin{split}
\notag
\max_{1 \leq m \leq M} \|\widehat{e}_m\|_{L^{\infty}(I_m)} \leq CM\varepsilon^{-1}_M \tt tol^{\star}.
\end{split}
\end{equation}
Since $\displaystyle \|u\|_{L^\infty(0,T)} = \log(\varepsilon_M^{-1})$, we remark that the error estimator blows up at a faster rate than the exact solution in this example as well. Moreover, Figure \ref{example2deltaEI} implies that this rate is again quasi-optimal in the best case scenario for the error estimator. To conclude, we remark that this result gives additional weight to the interpretation of $\widehat{\delta}_M$ as an \emph{a posteriori} approximation to the blow-up rate of the error on $(0,T)$ when $\F$ is nonlinear.

\begin{figure}
\centering
\includegraphics[scale=0.265]{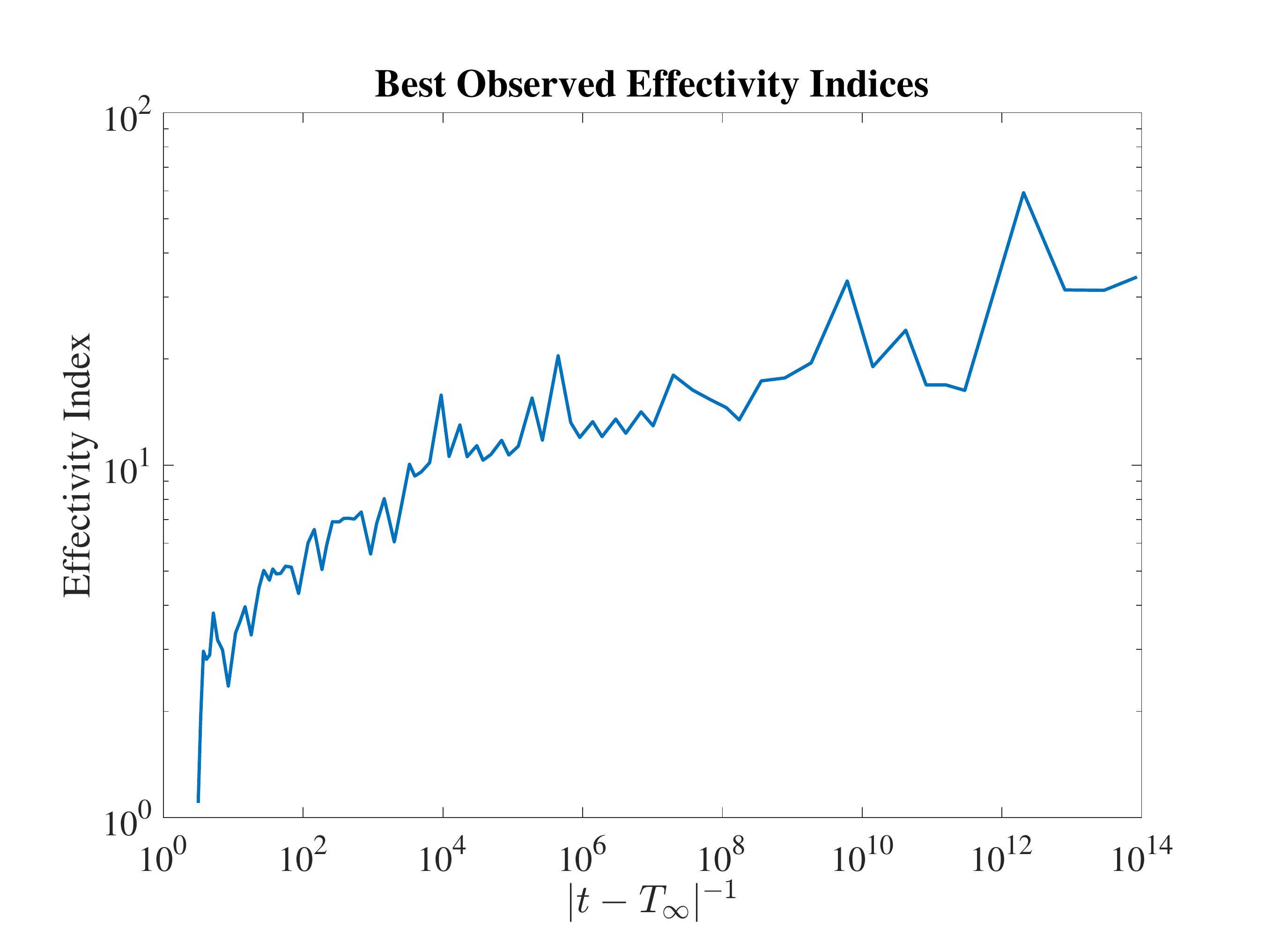} \includegraphics[scale=0.265]{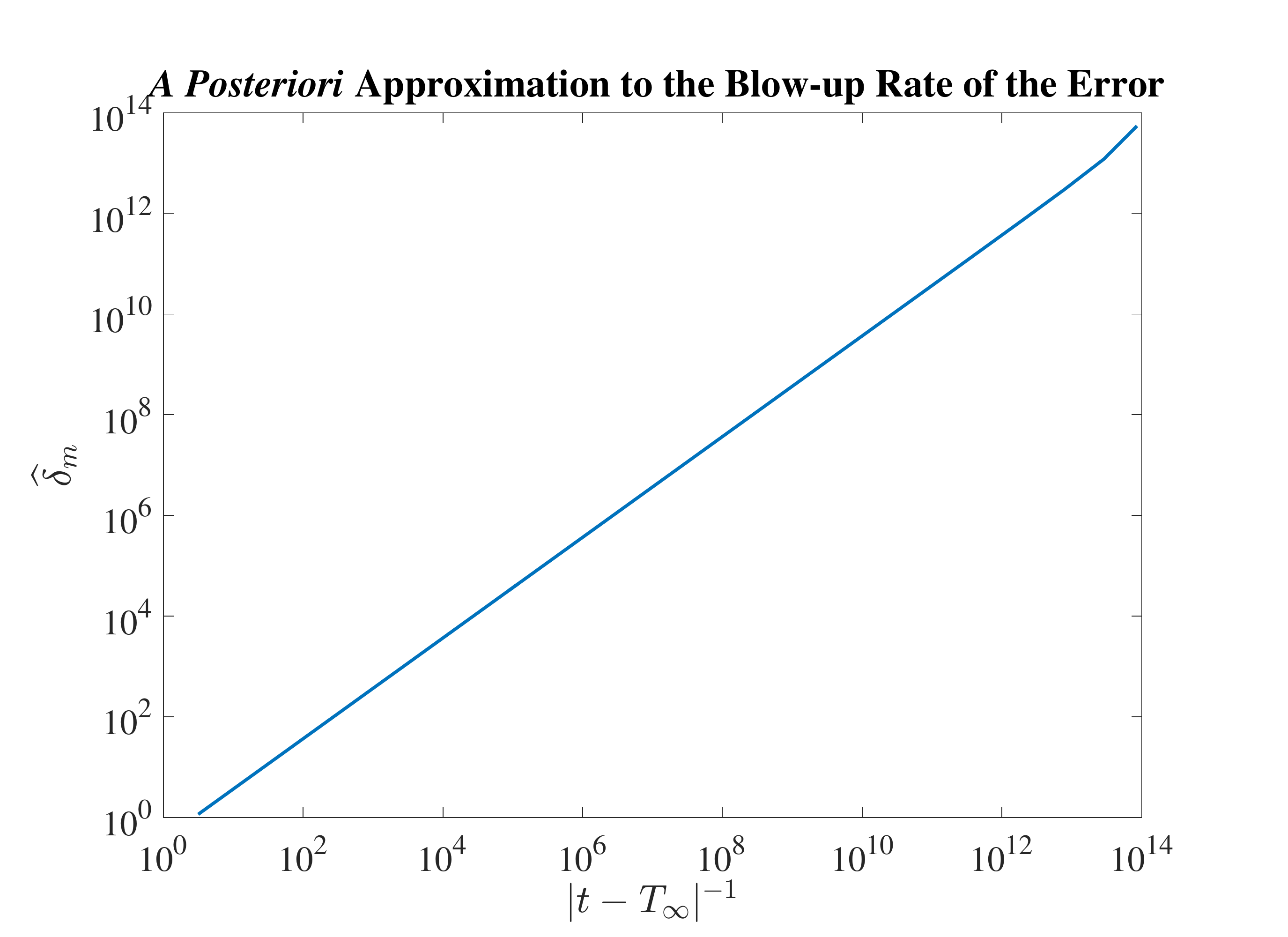}
\caption{Example 2: best observed effectivity indices and $\widehat{\delta}_m$.}
\label{example2deltaEI}
\end{figure}


\section{Conclusions}

In this paper, we derived a conditional \emph{a posteriori} error bound through a continuation argument for Galerkin time discretizations of general nonlinear initial value problems; the derived bound was then used to drive $h$ and $hp$ versions of an adaptive algorithm designed to approximate the blow-up time. Numerical experiments indicate that the $h$-version of the adaptive algorithm attains algebraic convergence towards the blow-up time while the $hp$-version of the adaptive algorithm attains exponential convergence towards the blow-up time. This work constitutes an important step towards deriving $hp$-version \emph{a posteriori} error bounds for the semilinear heat equation that are \emph{robust} with respect to the distance from the blow-up time, thereby extending the results in the works \cite{CGKM15, KM01}.


\bibliographystyle{amsplain}
\bibliography{paper}

\providecommand{\bysame}{\leavevmode\hbox to3em{\hrulefill}\thinspace}
\providecommand{\MR}{\relax\ifhmode\unskip\space\fi MR }
\providecommand{\MRhref}[2]{%
  \href{http://www.ams.org/mathscinet-getitem?mr=#1}{#2}
}
\providecommand{\href}[2]{#2}
\begin{thebibliography}{10}

\bibitem{AkrivisMakridakisNochetto:09}
G.~Akrivis, C.~Makridakis, and R.H. Nochetto, \emph{Optimal order a posteriori
  error estimates for a class of {R}unge-{K}utta and {G}alerkin methods},
  Numer.~Math. \textbf{114} (2009), no.~1, 133--160.

\bibitem{rannacher}
W.~Bangerth and R.~Rannacher, \emph{Adaptive finite element methods for
  differential equations}, Lectures in Mathematics, ETH Z\"urich, Birkh\"auser
  Verlag, Basel, 2003.

\bibitem{BrunnerSchoetzau06}
H.~Brunner and D.~Sch{\"o}tzau, \emph{$hp$-discontinuous {G}alerkin
  time-stepping for {V}olterra integro\-differential equations}, SIAM J. Numer.
  Anal. \textbf{44} (2006), 224--245.

\bibitem{CGKM15}
A.~Cangiani, E.H. Georgoulis, I.~Kyza, and S.~Metcalfe, \emph{Adaptivity and
  blow-up detection for nonlinear evolution problems}, to appear in SIAM
  J.~Sci.~Comput. (2017).

\bibitem{CM08}
E.~Cuesta and C.~Makridakis, \emph{A posteriori error estimates and maximal
  regularity for approximations of fully nonlinear parabolic problems in
  {B}anach spaces}, Numer. Math. \textbf{110} (2008), 257--275.

\bibitem{DelfourHagerTrochu81}
M.~Delfour, W.~Hager, and F.~Trochu, \emph{Discontinuous {G}alerkin methods for
  ordinary differential equations}, Math. Comp. \textbf{36} (1981), 455--473.

\bibitem{DeDu86}
M.~C. Delfour and F.~Dubeau, \emph{Discontinuous polynomial approximations in
  the theory of one-step, hybrid and multistep methods for nonlinear ordinary
  differential equations}, Math. Comp. \textbf{47} (1986), no.~175, 169--189,
  S1--S8.

\bibitem{Estep95}
D.~Estep, \emph{A posteriori error bounds, global error control for
  approximation of ordinary differential equations}, SIAM J. Numer. Anal.
  \textbf{32} (1995), 1--48.

\bibitem{EstepFrench94}
D.~Estep and D.~French, \emph{Global error control for the continuous
  {G}alerkin finite element method for ordinary differential equations}, RAIRO
  Mod\'el. Math. Anal. Num\'er. \textbf{28} (1994), 815--852.

\bibitem{FWW14}
T.~Fankhauser, T.~P. Wihler, and M.~Wirz, \emph{The hp-adaptive {FEM} based on
  continuous {S}obolev embeddings: {I}sotropic refinements},
  Comput.~Math.~Appl. \textbf{67} (2014), no.~4, 854--868.

\bibitem{JanssenWihler:15}
B.~Janssen and T.~P. Wihler, \emph{Existence results for the continuous and
  discontinuous {G}alerkin time stepping methods for nonlinear initial value
  problems}, arXiv preprint arXiv:1407.5520 (2014).

\bibitem{johnson}
C.~Johnson, \emph{Error estimates and adaptive time-step control for a class of
  one-step methods for stiff ordinary differential equations}, SIAM J. Numer.
  Anal. \textbf{25} (1988), 908--926.

\bibitem{KNS04}
D.~Kessler, R.~H. Nochetto, and A.~Schmidt, \emph{A posteriori error control
  for the {A}llen-{C}ahn problem: circumventing {G}ronwall'��s inequality},
  ESAIM Math. Model. Numer. Anal. \textbf{38} (2004), no.~1, 129--142 (eng).

\bibitem{KM01}
I.~Kyza and C.~Makridakis, \emph{Analysis for time discrete approximations of
  blow-up solutions of semilinear parabolic equations}, SIAM J. Numer. Anal.
  \textbf{49} (2011), no.~1, 405--426.

\bibitem{MM14}
W.F. Mitchell and M.A. McClain, \emph{A comparison of hp-adaptive strategies
  for elliptic partial differential equations}, ACM Transactions on
  Mathematical Software (TOMS) \textbf{41} (2014), no.~1, 2.

\bibitem{SchoetzauSchwabDGODE}
D.~Sch{\"o}tzau and C.~Schwab, \emph{An $hp$ a-priori error analysis of the
  {DG} time-stepping method for initial value problems}, Calcolo \textbf{37}
  (2000), 207--232.

\bibitem{SchoetzauSchwab00}
\bysame, \emph{Time discretization of parabolic problems by the $hp$-version of
  the discontinuous {G}alerkin finite element method}, SIAM J. Numer. Anal.
  \textbf{38} (2000), 837--875.

\bibitem{SchoetzauSchwab01}
\bysame, \emph{$hp$-discontinuous {G}alerkin time-stepping for parabolic
  problems}, C. R. Acad. Sci. Paris, S\'erie I \textbf{333} (2001), 1121--1126.

\bibitem{ScWi10}
D.~Sch{\"o}tzau and T.~P. Wihler, \emph{A posteriori error estimation for
  {$hp$}-version time-stepping methods for parabolic partial differential
  equations}, Numer. Math. \textbf{115} (2010), no.~3, 475--509. \MR{2640055
  (2012c:65165)}

\bibitem{Gerdes}
T.~Werder, K.~Gerdes, D.~Sch\"otzau, and C.~Schwab, \emph{$hp$-discontinuous
  {G}alerkin time-stepping for parabolic problems}, Comput. Methods Appl. Mech.
  Engrg. \textbf{190} (2001), 6685--6708.

\bibitem{Wihler05}
T.~P. Wihler, \emph{An a-priori error analysis of the $hp$-version of the
  continuous {G}alerkin {FEM} for nonlinear initial value problems}, J. Sci.
  Comput. \textbf{25} (2005), 523--549.

\end{thebibliography}
\end{document}